\documentclass[onefignum,onetabnum]{siamart190516}

\usepackage{amsmath, amsfonts, amssymb}
\usepackage{hyperref}
\usepackage{bm}
\usepackage{stmaryrd}
\usepackage{color}
\usepackage{url}
\input xy \xyoption{all}
\usepackage[all,poly]{xy}
\xyoption{color}
\usepackage{graphicx,epstopdf} 
\usepackage[caption=false]{subfig}
\usepackage{algorithmic}

\ifpdf
  \DeclareGraphicsExtensions{.eps,.pdf,.png,.jpg}
\else
  \DeclareGraphicsExtensions{.eps}
\fi

\newsiamremark{remark}{Remark}
\newsiamremark{hypothesis}{Hypothesis}
\newsiamremark{example}{Example}
\newsiamthm{assumption}{Assumption}
\newsiamthm{claim}{Claim}
\crefname{hypothesis}{Hypothesis}{Hypotheses}
\crefname{assumption}{Assumption}{Assumptions}
\crefname{remark}{Remark}{Remarks}
\crefname{example}{Example}{Examples}
\Crefname{ALC@unique}{Line}{Lines} 

\headers{Convergence of SALS for Random Tensors}{Y. Cao, S. Das, L. Oeding, H.-W. van Wyk}

\title{Analysis of the Stochastic Alternating Least Squares Method for the Decomposition of Random Tensors\thanks{Submitted to the editors 04/01/2020.
}}

\author{Yanzhao Cao\thanks{Department of Mathematics and Statistics, Auburn University, Auburn AL 
  (\email{yzc0009@auburn.edu}, \email{szd0041@auburn.edu}, \email{oeding@auburn.edu}, \email{hzv0008@auburn.edu}).}
\and Somak Das\footnotemark[2] \and  Luke Oeding\footnotemark[2] \and Hans-Werner van Wyk\footnotemark[2]}

\usepackage{amsopn}

\newcommand{\X}{\mathcal X}
\newcommand{\bX}{\bm{\mathcal{X}}}

\newcommand{\x}{{\bm x}}
\newcommand{\g}{{\bm g}}
\newcommand{\ba}{{\bm a}}
\newcommand{\F}{\mathcal{F}}

\newcommand{\cmin}{c_{\min}}
\newcommand{\cmax}{c_{\max}}
\newcommand{\E}{\mathbb{E}}
\newcommand{\R}{\mathbb{R}}

\newcommand{\db}[1]{\left\llbracket #1 \right\rrbracket}

\DeclareMathOperator*{\argmin}{argmin}

\newcommand{\ccirc}[1]{\xymatrix@1{*+<1ex>[o][F-]{#1}}}

\begin{document}
\maketitle

\begin{abstract}
Stochastic Alternating Least Squares (SALS) is a method that approximates the canonical decomposition of averages of sampled random tensors. Its simplicity and efficient memory usage make SALS an ideal tool for decomposing tensors in an online setting. We show, under mild regularization and readily verifiable assumptions on the boundedness of the data, that the SALS algorithm is globally convergent. Numerical experiments validate our theoretical findings and demonstrate the algorithm's performance and complexity. 
\end{abstract}



\section{Introduction}
Multi-modal data is represented by a tensor, or a multidimensional matrix. Tensor data is present in ares such as Natural Language Processing (NLP) \cite{bouchard-etal-2015-matrix}, Blind Source Separation (BSS) \cite{Bousse, SGB}, and Phylogenetic Tree Reconstruction (PTR) \cite{eriksson2005tree, allman2006identifiability, ishteva2013unfolding}. In each of these areas, canonical decomposition (CANDECOMP)/parallel factors (PARAFAC), also known as CP tensor decomposition (representing a given tensor as a sum of rank-one tensors), provides important insights since the components of the rank-one terms, the factors, represent meaning in the data. For NLP given a tensor constructed from large text databases, the rank-one terms could represent topics, and the factors could represent words. For BSS given a tensor constructed from signals arriving at a single receiver from unknown sources, the rank-one terms could represent sources, and the factors could be used to locate the sources. For PTR given a tensor constructed as a contingency table for instances of neucleotide combinations from aligned DNA from several species, the factors represent model parameters for the phylogenetic tree. 

In applications observations may represent samples of an underlying \emph{random tensor}. For example, the word co-occurence frequencies used in NLP may come from a sample of a collection of documents. Here fast, reliable algorithms are desired for the CP decomposition of the \emph{average tensor}. Because CP decomposition is known to be NP hard in general \cite{HillarLim09mosttensor}, and because the average tensor is usually only available through sampling, we investigate approximate decompositions based on stochastic optimization. In \cite{Maehara2016}, Maehara et al. propose a stochastic version of the widely-used alternating least squares (ALS) method, the stochastic alternating least squares (SALS) method, and show the algorithm's efficiency by means of numerical experiments. In this paper, we provide a detailed analysis of the algorithm, showing under mild regularization and a minimal set of verifiable assumptions that it is globally convergent to a local minimizer for any initial guess. In \Cref{section:numerical}, we also include a detailed discussion the algorithm's complexity and efficiency.

The alternating least squares (ALS) method, first proposed by Carroll and Chang \cite{Carroll1970}, remains the most widely used algorithm for computing the CP decomposition of a tensor \cite{Faber2003}. This block-nonlinear Gauss-Seidel method \cite{Bertsekas1989,Ortega2000} exploits the component-wise quadratic structure of the cost functional to compute iterates efficiently and with a low memory footprint. It has been modified \cite{Cheng2016} to exploit sparsity structure in data, which typically occurs in tensors representing co-occurence frequencies such as in the tree reconstruction above. Regularized versions of the ALS method, as well as the associated proximal minimization iterations considered in \cite{Li2013b}, help mitigate the potential ill-posedness of the CP problem to within sample noise.
Although one may compute the tensor's expectation \emph{a priori} and decompose it by means of the standard ALS method (see \cref{rem:equivalent_to_cp_of_expectation}), such an approach results in a loss of sparsity and cannot seemlessly acommodate the arrival of new samples. In \cite{Maehara2016}, Maehara et al. proposed the Stochastic Alternating Least Squares method, a block-stochastic minimization method that preserves the salient features of the ALS method, while also efficiently incorporating tensor samples in the optimization procedure.
Recent work (see \cite{Battaglino2018,Kolda2019}) has considered the related problem of using randomized (or stochastic) methods to decompose existing large-scale tensors. In particular, in \cite{Battaglino2018}, a variant of the ALS method is developed that approximates the component least squares problem efficiently by randomized/sketching methods. In \cite{Kolda2019}, a dense tensor is expressed as the expectation of a sequence of sparse ones, thereby allowing the use of stochastic gradient descent (SGD) methods.  

The convergence analysis of the SALS algorithm applied to the CP decomposition of a random tensor is complicated by the fact that the underlying cost functional is not convex. Moreover, the algorithm itself is a stochastic, block-iterative, second order method whose successive iterates do not have the Markovian dependence structure present in classical stochastic gradient descent (SGD) methods. The convergence of the ALS method was studied in \cite{Bezdek2003}, where a quotient-linear convergence rate was established. Block-coordinate techniques for the unconstrained optimization of general, possibly nonconvex, deterministic functionals were treated in \cite{Grippof1999} (see also \cite{Nesterov2012, Beck2013, Xu2017} and references therein). Xu and Yin \cite{Xu2015} discuss the convergence of block stochastic \emph{gradient} methods for averages of convex and nonconvex functionals. They rely on assumptions (such as the uniformly bounded variance of gradient iterates) that, while standard in the literature (see e.g. \cite{Bottou2018}), are difficult to verify in practice since they pertain to the algorithm itself.  The main contributions of this paper are to prove the convergence of the SALS algorithm, a block-stochastic \emph{Newton} method, for the CP decomposition of the average of a \emph{random} tensor, subject to a single verifiable assumption relating to the boundedness of the observed data.

This paper is structured as follows. In \Cref{section:problem}, we introduce the CP decomposition problem for random tensors and describe the stochastic alternating least squares algorithm (\cref{alg:stochastic_als}). \Cref{section:convergence} contains our main theoretical contributions. In \Cref{subsection:regularity}, we exploit the special multinomial structure of the cost functional to quantify the regularity of its component-wise gradient and Hessian in terms of the size of the algorithm's iterate vectors (see \cref{lem:compgrad_lipschitz} and its corollaries). In \Cref{subsection:boundedness} we show that the iterates themselves can be bounded in terms of the random tensor to be decomposed (\cref{lem:iterates_bnded}), which naturally leads to our single, verifiable assumption on the latter's statistical properties (\cref{ass:tensor_bounded}). In \Cref{subsection:convergence_proof}, we show that the iterates generated by the SALS algorithm converge to a local minimizer. We validate the SALS method numerically via a few computational examples in \Cref{section:numerical} and offer concluding remarks in \Cref{section:conclusion}.

\section{Notation and Problem Description} \label{section:problem}
We to follow notational conventions that are closely aligned with the literature on the CP decomposition (see \cite{Kolda2009}), as well as on stochastic optimization (see \cite{Bottou2018,Xu2015}). Use use lower case letters to denote scalars, bold letters for vectors, uppercase letters for matrices, and uppercase calligraphic letters for tensors. We use superscripts to indicate iteration (or sub-iteration) indices and subscripts for components, i.e. $x_i^k$ is the $i$-th component at the $k$-th iteration. Multi-indices are denoted by bold letters and sums, products, and integrals over these should be interpreted to be in iterated form. For the block coordinate minimization method in \cref{alg:stochastic_als}, it is convenient to express the vector $\x=(\x_1,...,\x_p)$ in terms of its $i$-th block component $\x_i$ and the remaining components, denoted  $\x_{i^*}:=(\x_1,\hdots,\x_{i-1},\x_{i+1},\hdots,\x_p)$. Using this notation, we write $\x=(\x_i,\x_{i^*})$ with the implicit understanding that the block components are in the correct order. The Frobenius norm of a vector, matrix, or tensor, i.e. the root mean square of its entries, is denoted by $\|\cdot \|$. For  $1\leq p \leq \infty$, $\|\cdot\|_p$ denotes the standard Euclidean $p$-norm for vectors and the induced matrix $p$-norm for matrices. We use `$\circ$' to denote the outer product, `$*$' for the component-wise (or Hadamard) product, `$\odot$' for the column-wise Khatri-Rao product, and `$\otimes$' for the Kronecker product.

Let $(\Omega, \mathcal F, d\mu)$ be a complete probability space and let $\X: \Omega \rightarrow \R^{n_1\times n_2\hdots \times n_p}$ be a measurable map, also known as a $p$-th  order random tensor. In practice the underlying probability space is rarely known explicitly. Instead, its law can often be observed indirectly through sample realizations $\X^1,\X^2,\hdots,\X^n$ of $\X$ that are assumed to be independent and identically distributed (iid). We use $\E$ to denote expectation:
\[
\E\left[f\right] = \int_\Omega f(\X)d\mu_\X
\]
for any integrable function $f:\R^{n_1\times \hdots \times n_p} \rightarrow \R^{m}$. 

The rank-$r$  decomposition problem for $\X$ (or its sample realizations) amounts to finding a set of $r$ rank-one deterministic tensors $\{\hat \X_j\}_{j=1}^r$,  so that the quantity 
\begin{equation}\label{eq:rank_r_approximation}
\hat \X := \sum_{j=1}^r \hat \X_j
\end{equation} 
is a good overall representation of $\X$. Each rank-one tensor $\hat \X_j$, $j=1,\hdots,r$, is formed by the outer product $\hat \X_j = {\bm a}_{1,j}\circ \hdots \circ {\bm a}_{p,j}$ where ${\bm a}_{i,j} = (a_{i,j,1},\hdots,a_{i,j,n_i})\in \R^{n_i}$ for each $i=1,\hdots p$, i.e. $\hat \X_{j}\in \R^{n_1\times \hdots \times n_p}$ is defined component-wise by 
\[
\hat \X_{j,i_1,\hdots,i_p}=\prod_{l=1}^p a_{l,j,i_l}.
\] 
We use the mean squared error $\E\left[\|\X - \hat \X\|^2 \right]$ to measure the quality of the representation. Other risk measures may be more appropriate, depending on the application, possibly requiring a different analysis and approximation technique.

For analysis and computation, it is convenient to consider two other representations of the design variable. We define the $i$-th factor matrix $A_i = [{\bm a}_{i,1},\hdots,{\bm a}_{i,r}]$ in $\R^{n_i\times r}$ to consist of the $i$-th component vectors of each term in decomposition \eqref{eq:rank_r_approximation}. Let ${\bm x}=\mathrm{vec}([A_1,\hdots,A_r])\in \R^{nr}$, with $n=\sum_{i=1}^p n_i$, be the vectorization of the factor matrices, i.e. the concatenation of their column vectors. We also write $\x$ in block component form as $\x = (\x_1,\hdots,\x_p)$, where $\x_i = \mathrm{vec}(A_i)\in \R^{rn_i}$ for $i=1,\hdots,p$. To emphasize dependence of $\hat \X$ on ${\bm x}$, we write the following (which also defines $\db{\cdot}$)
\[
\hat \X = \db{\x} = \db{A_1,\hdots,A_p} := \sum_{j=1}^r {\bm a}_{1,j}\circ \hdots \circ {\bm a}_{p,j}.
\]
No \emph{efficient} closed form solution of the factorization problem exists (even for deterministic tensors)  \cite{HillarLim09mosttensor}.
So it is commonly reformulated as the optimization problem 
 
\begin{equation}\label{eq:random_tensor_decomposition}
\min_{{\bm x}\in \R^{nr}} \E\left[\|\X-\db{\x}\|^2\right].
\end{equation}

\begin{remark}\label{rem:equivalent_to_cp_of_expectation}
Letting ${\bm i} = (i_1,\hdots,i_p)$ and ${\bm n} = (n_1,\hdots,n_p)$, it follows readily that
\begin{align*}
\E\left[ \left\|\X-\llbracket \x\rrbracket\right\|^2\right] &= 
\E\left[ \sum_{{\bm i}=1}^{\bm n} \left(\X_{\bm i} - \llbracket {\bm x}\rrbracket_{\bm i}\right)^2 \right] \\&= \sum_{{\bm i}=1}^{{\bm n}} \E\left[\X_{\bm i}^2\right] - 2 \E\left[\X_{\bm i}\right]\llbracket {\bm x}\rrbracket_{\bm i} + \llbracket {\bm x}\rrbracket_{\bm i}^2 \\
&= \sum_{{\bm i}=1}^{\bm n} \mathrm{var}(\X_{\bm i}) + \sum_{{\bm i}=1}^{\bm n} \left(\E(\X_{\bm i})-\llbracket {\bm x}\rrbracket_{\bm i}\right)^2.
\end{align*}
Since the variance $\mathrm{var}(\X_{{\bm i}})$ of $\X_{{\bm i}}$ is constant in the design variable $\bm x$, the minimization Problem \eqref{eq:random_tensor_decomposition} is equivalent to the decomposition of the expected tensor, i.e.
\[
 \min_{x\in \R^{nr}}  \frac{1}{2}\|\E[\X]-\llbracket {\bm x} \rrbracket \|^2.
\]
The analysis in this work is however based on the formulation given by \eqref{eq:random_tensor_decomposition}, since it lends itself more readily to extensions to more general risk functions. 
\end{remark}

Tensor decompositions have scale ambiguity. Indeed, for any direction $i=1,..,p$ and any scalars $\beta_{i,1},\hdots,\beta_{i,r-1}>0$, let $\beta_{i,r}=1/\prod_{j=1}^{r-1}\beta_{i,j}$. Then
\[
\sum_{j=1}^r (\beta_{1,j}{\bm a}_{1,j})\circ \hdots \circ (\beta_{p,j}{\bm a}_{p,j}) = \sum_{j=1}^r{\bm a}_{1,j}\circ \hdots \circ {\bm a}_{p,j}.
\]
The optimizer of \eqref{eq:random_tensor_decomposition} therefore lies on a simply connected manifold (see \cite{Uschmajew2012}), which can lead to difficulties in the convergence of optimization algorithms. To ensure isolated minima that can be readily located, it is common to enforce bounds on the size of the factors, either directly through an appropriate normalization (see \cite{Uschmajew2012}), or indirectly through the addition of a regularization term \cite{Maehara2016}. We pursue the latter strategy, leading to the problem

\begin{multline}
\label{eq:regularized_random_tensor_decomposition}
\min_{{\bm x}\in \R^{nr}} F({\bm x}) = \min_{{\bm x}\in \R^{nr}}  \E\left[f(\x)\right],
\text{ with}
\\
f(\x;\X) = \frac{1}{2}\left\|\X - \llbracket {\bm x} \rrbracket \right\|^2 + \frac{\lambda}{2}\|{\bm x}\|^2, \qquad \x\in \R^{nr}, \lambda>0.
\end{multline}

While the regularization term biases the minimizers of Problem \eqref{eq:random_tensor_decomposition} its inclusion is key to guaranteeing well-posedness in the presence of noise. It plays a pivotal role in (i) proving the existence of a minimizer (\cref{lem:existence}), (ii) ensuring the positive-definiteness of the Hessian (\cref{cor:hessian_posdef}), and (iii) guaranteeing the boundedness of iterates in terms of random tensor $\X$ (\cref{lem:iterates_bnded}). Heuristic methods are typically used to choose the value of $\lambda$ that balances the bias of the optimizers against stability considerations, the most well-known of which is the Morozov discrepancy principle \cite{Morozov1984}. 

\begin{lemma}[Existence of Minimizers]\label{lem:existence}
Problem \eqref{eq:regularized_random_tensor_decomposition} has at least one minimizer.
\end{lemma}
\begin{proof}
Let $F^* = \inf_{{\bm x}\in \R^{nr}} F({\bm x})$. So there exists a sequence $\{{\bm x^k}\}_{k=1}^\infty$ in $\R^{nr}$ with
\[
\lim_{k\rightarrow \infty} F({\bm x}^k) = F^*,
\] 
from which it follows that $F({\bm x^k})$ is bounded. The inequality 
\begin{equation}\label{eq:bound_normx_by_F}
\|{\bm x}\|^2 \leq \frac{2}{\lambda}F({\bm x}) \ \ \text{for all } {\bm x}\in \R^{rn} ,
\end{equation} 
allows us to establish the boundedness of $\{\x^k\}_{k=1}^\infty$. By compactness, there exists a convergent subsequence ${\bm x}^{k_i}\rightarrow {\bm x}^*$ as $i\rightarrow \infty$. The continuity of $F$ then implies
\[
F({\bm x^*}) = \lim_{i\rightarrow \infty} F({\bm x}^{k_i}) = F^*.
\]
\end{proof}

\begin{remark}
For the sample realizations $f(\x;\X)$ we have a bound similar to \eqref{eq:bound_normx_by_F}: 
\begin{equation}\label{eq:bound_normx_by_f}
\text{for all } \x \in \R^{nr}, \quad \|\x\|^2 \leq \frac{2}{\lambda} f(\x;\X), \ \ \ \text{a.s. on } \Omega.
\end{equation}
\end{remark}

\subsection{The Stochastic Alternating Least Squares Method}

Although the objective function $F$ is a high degree polynomial in general, it is quadratic in each component factor matrix $A_i$. This is apparent from the matricized form of \eqref{eq:regularized_random_tensor_decomposition}. Recall \cite{Khatri1968} that the columnwise Khatri-Rao product $\odot: \R^{n_i\times r}\times \R^{n_j\times r}\rightarrow \R^{n_in_j\times r}$ of matrices $A=[\bm{a}_1,...,{\bm a}_r]$ and $B = [{\bm b}_1,...,{\bm b}_r]$ is defined as their columnwise Kronecker product, i.e. $A\odot B = [{\bm a}_1 \otimes {\bm b}_1,...,{\bm a}_r \otimes {\bm b}_r]$. 
The matricization $\llbracket A_1,\hdots,A_p \rrbracket_{(i)}$ of the rank $r$ tensor $\llbracket A_1,\hdots,A_p \rrbracket$ along the $i$-th fiber takes the form (see \cite{Acar2011})
\begin{equation}\label{eq:theta_definition}
\llbracket A_1,\hdots,A_p \rrbracket_{(i)} = A_i\left(A_p \odot \hdots \odot A_{i+1} \odot A_{i-1} \odot \hdots \odot A_{1}\right)^T=:A_i \Theta_i^T,
\end{equation}
where $\Theta_i$ is defined to be the iterated Khatri-Rao product on the right. Note that $\Theta_i$ does not depend on $A_i$, and hence the matricized tensor decomposition is linear in $A_i$. Since the Frobenius norm is invariant under matricization, the sample objective function $f(\x,\X)$ can then be rewritten as a quadratic function in $A_i$, i.e.
\[
f(A_1,\hdots,A_p;\X) := \frac{1}{2}\left\|\X_{(i)} - A_i \Theta_i^T \right\|^2 + \frac{\lambda}{2} \sum_{j=1}^p\|A_j\|^2.
\]
Vectorizing this form yields a linear least squares objective function in $\x_i$, namely
\[
f(\x_1,\hdots,\x_p;\X) = \frac{1}{2}\left\|\mathrm{vec}(\X_{(i)}) - (\Theta_i \otimes I_{n_i}) \x_i \right\|^2 + \frac{\lambda}{2}\sum_{j=1}^p\|\x_{j}\|^2,
\]
whose component-wise gradient and Hessian are given respectively by
\begin{align}\label{eq:componentwise_stochastic_gradient}
\nabla_{\x_i}f(\x;\X) &= -(\Theta_i^T\otimes I_{n_i})\mathrm{vec}(\X_{(i)}) + \left((\Theta_i^T\Theta_i+\lambda I_{r})\otimes I_{n_i}\right)\x_{i}, \quad \text{and}
\\
\label{eq:componentwise_stochastic_hessian}
\nabla_{\x_{i}}^2 f(\x) &= (\Theta_i^T\Theta_i+\lambda I_{r})\otimes I_{n_i},
\end{align}
where $I_r \in \R^{r\times r}$ and $I_{n_i} \in \R^{n_i\times n_i}$ are identity matrices. In the presence of regularization, the Hessian matrix \eqref{eq:componentwise_stochastic_hessian} is strictly positive definite with lower bound that is independent of both $\X$ and of the remaining components $\x_{i^*}$ (see \cref{cor:hessian_posdef}). Consequently, each sampled component-wise problem
\[
\min_{\x_i \in \R^{rn_i}} f(\x;\X) 
\]
has a unique solution given by the stationary point $\x_i$ satisfying $\nabla_{\x_{i}}f(\x,\X)=0$. It is more efficient  to use the matricized form of the stationarity condition, i.e.
\begin{equation}\label{eq:sample_fonop}
0 = -\X_{(i)} \Theta_i + A_i (\Theta_i^T\Theta_i + \lambda I_r). 
\end{equation}
For any $A\in \R^{n_i\times r}, B\in \R^{n_j\times r}$, it can be shown \cite{Smilde2004} that
\[
(A\odot B)^T(A\odot B) = (A^TA) * (B^TB),
\]
where $*$ denotes the (Hadamard) product. Repeatedly using this identity, we have
\begin{equation}
\Theta_i^T \Theta_i := (A_1^TA_1) * \hdots * (A_{i-1}^TA_{i-1}) * (A_{i+1}^T A_{i+1}) * \hdots * (A_p^TA_p),  \label{eq:theta_transpose_theta}
\end{equation}
so the Hessian can be computed as the component-wise product of $p$ matrices in $\R^{r\times r}$. 

The (deterministic) ALS method (\cref{alg:als}) exploits the component-wise quadratic structure of the objective functional $F$, which it inherits from $f$. In the $k$-th block iteration, the ALS algorithm cycles through the $p$ components $\x_1,\hdots,\x_p$ of $\x$, updating each component in turn in the direction of the component-wise minimizer. The function is also updated at each subiteration to reflect this change.  Specifically, the iterate at the beginning of the $k$-th block is denoted by
\begin{align*}
\x^k = \x^{k,0} = (\x_1^k,\x_2^k,...,\x_p^k). 
\end{align*}
The ALS algorithm then generates a sequence of sub-iterates $\x^{k,1},...,\x^{k,p}$, where
\[
\x^{k,i} = (\x_1^{k+1},\hdots,\x_{i-1}^{k+1},\x_i^{k+1}, \x_{i+1}^{k},\hdots,\x_p^k).
\] 
Note that, under this convention, $\x^{k,p}=\x^{k+1}=\x^{k+1,0}$.

\begin{algorithm}
\caption{The Alternating Least Squares Algorithm}
\label{alg:als}
\begin{algorithmic}[1]
\STATE{Initial guess ${\bm x^1}$}
\FOR{$k=1,2,\hdots$}
\FOR{$i=1,\hdots,p$}
\STATE{$\displaystyle \x_i^{k+1} = \argmin_{\x_i\in \R^{n_i}} F(\x_1^{k+1},\hdots,\x_{i-1}^{k+1},\x_i, \x_{i+1}^k,\hdots,\x_p^k)$}
\STATE{${\bm x}^{k,i}=({\bm x}_1^{k+1},\hdots,{\bm x}_i^{k+1},{\bm x}_{i+1}^k,\hdots,{\bm x}_p^k)$}
\ENDFOR
\ENDFOR
\end{algorithmic}
\end{algorithm}

Although the descent achieved by the ALS method during $k$-th block iteration is most likely not as large as a descent would be for a monolithic descent direction for $F$ over the entire space $\R^{nr}$, the ALS updates can be obtained at a significantly lower computational cost and with a lower memory footprint. 

The block-component-wise quadratic structure of the integrand $F$ allows us to compute $\x_i^{k+1}$ explicitly. To this end we form the second order Taylor expansion $F$ about $\x_i$ at the current iterate $\x^{k,i-1}$, i.e.
\begin{align}\label{eq:expected_cost_taylor}
F(\x_1^{k+1},\hdots\x_{i-1}^{k+1},\x_i+{\bm p},\x_{i+1}^k,\hdots,\x_p^{k}) &= F^{k,i} + (\g^{k,i})^T {\bm p}+ \frac{1}{2} {\bm p}^T H^{k,i} {\bm p},
\\\nonumber
&\hspace{-13em} \text{with}\quad F^{k,i} = F(\x^{k,i-1}) = \E\left[f(\x^{k,i-1})\right],
\\\nonumber
&\hspace{-10em}{\bm g}^{k,i} = \nabla_{\x_i} F(\x^{k,i-1}) = \E\left[\nabla_{\x_i} f(\x^{k,i-1})\right], \quad \text{and}
\\\nonumber
&\hspace{-10em}H^{k,i} = \nabla_{\x_i}^2 F(\x^{k,i-1}) = \nabla_{\x_i}^2 f(\x^{k,i-1}).
\end{align}
Note that $H^{k,i}$ does not depend explicitly on $\X$ and is therefore a deterministic quantity for known $\x$. The component-wise minimizer $\x_i^{k+1}$ of $F$ can be calculated explicitly as the Newton update 
\[
\x_i^{k+1} = \x_i^k - (H^{k,i})^{-1}{\bm g}^{k,i}.
\]

Although both the design variable and the functions appearing in Problem \eqref{eq:regularized_random_tensor_decomposition} are deterministic, any gradient based optimization procedure, such as the ALS method described above, requires the approximation of an expectation at every iterate, the  computational effort of which is considerable. This suggests the use of stochastic gradient sampling algorithms that efficiently incorporate the sampling procedure into the optimization iteration, while exploiting data sparsity.

The stochastic gradient descent method \cite{Robbins1951} addresses the cost of approximating the expectation by computing descent directions from small sampled batches of function values and gradients at each step of the iteration. To accommodate the noisy gradients, the step size is reduced at a predetermined rate. For the stochastic alternating least squares (SALS) method (\cref{alg:stochastic_als}), we determine the sample at the beginning of the $k$-th block iteration. For a batch $\bX = (\X^1,\hdots,\X^m)$ of $m$ iid random samples of $\X$, we write the component-wise batch average $\tilde{f}$ of $f$ as 
\[
\tilde{f}(\x;\bX) = \frac{1}{m}\sum_{l=1}^m f(\x;\X^l).
\]
As before, we can compute the component-wise minimizer
\begin{equation}\label{eq:componentwise_minimizer}
\hat \x_i^{k+1} = \argmin_{\x_i \in \R^{n_i}} \tilde f(\x_{1}^{k+1},\hdots\x_{i-1}^{k+1},\x_i,\x_{i+1}^{k},\hdots,\x_p^k ;\bX)
\end{equation}
at the $i$-th subiteration by means of the Newton step. To this end we express $\tilde f$ in terms of its second order Taylor expansion about the current iterate $\x^{k,i-1}$, i.e.
\begin{align*}
\tilde f(\x_1^{k+1},\hdots,\x_{i-1}^{k+1},\x_i+{\bm p},\x_{i+1}^k,\hdots,\x_p^{k};\bX) = \tilde f^{k,i} + (\tilde \g^{k,i})^T {\bm p}+ \frac{1}{2} {\bm p}^T H^{k,i} {\bm p}, \quad\text{with}
\\
\tilde f^{k,i}(\bX) = \frac{1}{m}\sum_{l=1}^m f(\x^{k,i-1};\X^l), \quad \text{and}\quad
\tilde {\bm g}^{k,i}(\bX) = \frac{1}{m}\sum_{l=1}^m\nabla_{\x_i} f(\x^{k,i-1};\X^l).
\end{align*}
The sample minimizer is thus
\[
\hat \x_i^{k+1} = \x_i^k - \left(H^{k,i}\right)^{-1}\tilde \g^{k,i}.
\]
To mitigate the effects of noise on the estimate, especially during later iterations, we modify the update by introducing a variable stepsize parameter $\alpha^{k,i}>0$, so that 
\begin{equation}\label{eq:x_update_sals}
\x_i^{k+1} = \x_i^k - \alpha^{k,i} \left(H^{k,i}\right)^{-1}\tilde \g^{k,i}.
\end{equation}

\begin{algorithm}
\caption{Stochastic Alternating Least Squares Algorithm}
\label{alg:stochastic_als}
\begin{algorithmic}[1]
\STATE{Initial guess ${\bm x^1}$}
\FOR{$k=1,2,\hdots$}
\STATE{Generate a random sample $\bX^k = [\X^{k,1},\hdots,\X^{k,m_k}]$}
\FOR{$i=1,\hdots,p$}
\STATE{Compute sample gradient $\tilde {\bm g}^{k,i}$ and Hessian $H^{k,i}$}
\STATE{Compute step size $\alpha^{k,i}$}
\STATE{Update $i$-th block ${\bm x}_{i}^{k+1}={\bm x}_{i}^{k} - \alpha^{k,i}\left(H^{k,i}\right)^{-1}\tilde{\bm g}^{k,i}$ so that \newline ${\bm x}^{k,i}=({\bm x}_1^{k+1},\hdots,{\bm x}_i^{k+1},{\bm x}_{i+1}^k,\hdots,{\bm x}_p^k)$}
\ENDFOR
\STATE{$\x^{k+1} = \x^{k,p}$}
\ENDFOR
\end{algorithmic}
\end{algorithm}

It is well known (see e.g. \cite{Bottou2018}) that a stepsize $\alpha^{k,i}$ that decreases at the rate of $O\left(\frac{1}{k}\right)$ as $k\rightarrow \infty$ leads to an optimal convergence rate for stochastic gradient methods. Here, we specify that the stepsize takes the form $\alpha^{k,i} = \frac{c^{k,i}}{k}$, where $c^{k,i}$ is a strictly positive, bounded sequence, i.e. there are constants $0<\cmin<\cmax\leq 2$ such that
\[
0<\cmin\leq c^{k,i} \leq \cmax, \qquad \text{for all } k=1,2,\hdots, \ i=1,2,\hdots,p.
\] 

One of the difficulties in analyzing the convergence behavior of stochastic sampling methods arises from the fact that the iterates $\x^{k,i}$ constitute a stochastic process generated by a stochastic algorithm that depends on the realizations of the sample batches $\bX^1,\bX^2,\hdots\bX^{k-1}$. Consequently, even deterministic functions, such as $F$, become stochastic when evaluated at these points, e.g.\! $F(\x^{k,i})$ is a random quantity. Moreover, successive iterates are statistically dependent, since later iterates are updates of earlier ones. Specifically, let $\mathcal F^{k} = \sigma(\bX^1, \hdots, \bX^k)$ be the $\sigma$-algebra generated by the first $k$ sample batches. At the end of the $i$-th subiteration in the $k$-th block, the first $i$ components of $\x^{k,i}=(\x^{k+1}_1,...,\x_i^{k+1},\x_{i+1}^k,...,\x_p^k)$ have been updated using $\bX^k$ and are thus $\F^{k}$-measurable, whereas the remaining components $\x_{i+1}^k,...,\x_p^k$ depend only on $\bX^1,...,\bX^{k-1}$ and are therefore only $\mathcal{F}^{k-1}$-measurable. To separate the effect of $\bX^k$ from that of $\bX^1,\hdots,\bX^{k-1}$ on an $\F^k$-measurable random variable $h$, it is often useful to invoke the law of total expectation, i.e.
\[
\E_{\bX^1,...,\bX^{k-1},\bX^k}[h] = \E_{\bX^1,\hdots,\bX^{k-1}}\left[\E_{\bX^k}\left[h|\F^{k-1}\right]\right].
\]

\section{Convergence Analysis}\label{section:convergence}

Here we discuss the convergence of \cref{alg:stochastic_als}. Since Problem \eqref{eq:regularized_random_tensor_decomposition} is nonconvex, there can in general be no unique global minimizer. \Cref{thm:convergence} establishes the convergence of the SALS iterates $\x^{k,i}$ to a local minimizer in expectation. Yet, the special structure of the CP problem allows us to forego most of the standard assumptions made in the SGD framework. In fact, we only assume boundedness of the sampled data $\X$. Indeed, we show that the regularity of the component-wise problem, as the Lipschitz constant of the gradient, depends only on the norm of the current iterate $\x^{k,i}$, which is in turn bounded by $\|\X\|$ (\cref{lem:iterates_bnded}).

\subsection{Regularity Estimates}\label{subsection:regularity}

In the following we exploit the multinomial structure of the cost functional to establish component-wise regularity estimates, such as local Lipschitz continuity of the sampled gradient $\nabla_{\x_i}f(\x,\X)$ and of the Newton step $-\left(\nabla^2_{\x_i}f(\x)\right)^{-1}\nabla_{\x_i}f(\x,\X)$, as well as bounded invertibility of the Hessian $\nabla^2_{\x_i}f(\x)$. 

In the proofs below, we will often need to bound a sum of the form $\sum_{i=1}^p \|\x_i\|^q$ in terms of $\|\x\|$, where $\x = (\x_1,...,\x_p)$. For this we use the fact that for finite dimensional vector spaces all norms are equivalent. In particular, for any norm subscripts $s$ and $t$, there exist dimension-dependent constants $0<c_{s,t}<c_{t,s}<\infty$ so that
\begin{equation}
c_{s,t}\|\x\|_{s} \leq \|\x\|_{t} \leq c_{t,s}\|\x_i\|_{s}.
\end{equation} 
Using norm equivalence with $s=2$ and $t=q$, we obtain 
\begin{equation}\label{eq:norm_equivalence}
\sum_{i=1}^p \|\x_i\|^q \leq \sum_{i=1}^p (c_{2,q}^i)^q \|\x_i\|_q^q  \leq (\bar c_{2,q})^q \|\x\|_q^q \leq \left(\frac{\bar c_{2,q}}{c_{q,2}}\right)^q \|\x\|^q,
\end{equation}
where $\displaystyle \bar c_{2,q} = \max_{i=1,...,p} c_{2,q}^i$, while $c_{2,q}^i$ and $c_{q,2}$ are the appropriate constants in the norm equivalence relation above for $\x_i$ and $\x$ respectively.

The estimates established in this section all follow from the local Lipschitz continuity of the mapping $\x_i^*\mapsto \Theta_i$, shown in \cref{lem:difference_of_thetas} below.
\begin{lemma}\label{lem:difference_of_thetas}
Let $\Theta_i$ and $\tilde \Theta_i$ be matrices defined in terms of $\x_{i^*}$ and $\tilde \x_{i^*}$ respectively via \eqref{eq:theta_definition}. Then 
\begin{align}\label{eq:difference_of_thetas}
\|\Theta_i - \tilde \Theta_i\| \leq P_\Theta(\|\x_{i^*}\|,\|\tilde \x_{i^*}\|)  \|\x_{i^*}-\tilde \x_{i^*}\|,
\quad \text{with}
\\ \nonumber
P_\Theta(\|\x_{i^*}\|,\|\tilde \x_{i^*}\|) = C \left( \|\x_{i^*}\|^{p-2}+\|\tilde \x_{i^*}\|^{p-2}\right)
\end{align}
for some constant $0\leq C<\infty$.
\end{lemma}
\begin{proof}
First note that for any  $A\in \R^{n_i\times r}, B\in \R^{n_j\times r}$, properties of Kronecker products, the Cauchy-Schwartz inequality, and Inequality \eqref{eq:norm_equivalence} imply that
\begin{align}
\|A\odot B\|^2 &= \|[\ba_1\otimes \bm{b}_1,\hdots,\ba_r\otimes \bm b_r]\|^2 = \sum_{j=1}^r \|\bm a_j\otimes \bm b_j\|^2 \nonumber \\
&= \sum_{j=1}^r \|\bm a_j\|^2 \|\bm b_j\|^2 \leq \left(\sum_{j=1}^r \|\bm a_j\|^4\right)^{\frac{1}{2}}\left(\sum_{j=1}^r \|\bm b_j\|^4\right)^{\frac{1}{2}} \leq \left(\frac{\bar c_{2,4}}{c_{4,2}}\right)^2 \|A\|^2 \|B\|^2.\label{eq:khatri_rao_norm_inequality}
\end{align}
Let $\bm k = (p,p-1,\hdots,i+1,i-1,\hdots,1)$ denote the indices in the Khatri-Rao product \eqref{eq:theta_definition}. Addition and subtraction of the appropriate cross-terms provides
\[
\Theta_i - \tilde \Theta_i = \bigodot_{j=1}^{p-1}A_{k_j} - \bigodot_{j=1}^{p-1}\tilde A_{k_j} = \sum_{j=1}^{p-1} \bigodot_{j'=1}^{j-1}\tilde A_{k_{j'}} \odot (A_{k_j}-\tilde A_{k_j})\odot \bigodot_{j'=j+1}^{p-1}A_{k_{j'}}.
\]
By \eqref{eq:khatri_rao_norm_inequality}, each term in the sum above can be bounded in norm by 
\begin{align*}
& \left \|\bigodot_{j'=1}^{j-1}\tilde A_{k_{j'}} \odot (A_{k_j}-\tilde A_{k_j})\odot \bigodot_{j'=j+1}^{p-1}A_{k_{j'}}\right \| \\
\leq & \left(\frac{\bar c_{2,4}}{c_{4,2}}\right)^p\left(\prod_{j'=1}^{j-1}\|\tilde A_{k_{j'}}\|\right) \left(\prod_{j'=j+1}^{p-1}\|A_{k_{j'}}\|\right) \|A_{k_j}-\tilde A_{k_j}\|\\
= & \left(\frac{\bar c_{2,4}}{c_{4,2}}\right)^p\left(\prod_{j'=1}^{j-1}\|\tilde \x_{k_{j'}}\|\right) \left(\prod_{j'=j+1}^{p-1}\|\x_{k_{j'}}\|\right) \|\x_{k_j}-\tilde \x_{k_j}\|.
\end{align*} 

By the arithmetic-geometric mean inequality, and the bound in \eqref{eq:norm_equivalence}, we have
\begin{align*}
\left(\prod_{j'=1}^{j-1}\|\tilde \x_{k_{j'}}\|\right)\left(\prod_{j'=j+1}^{p-1}\|\x_{k_{j'}}\|\right) & \leq \frac{1}{(p-2)}\left( \sum_{j'=1}^{j-1}\|\tilde \x_{k_{j'}}\|^{p-2} + \sum_{j'=j+1}^{p-1}\|\x_{k_{j'}}\|^{p-2}\right) \\
& \leq \frac{1}{p-2}\left( \sum_{j'=1}^{p-1}\|\tilde \x_{k_{j'}}\|^{p-2} + \sum_{j'=1}^{p-1}\|\x_{k_{j'}}\|^{p-2}\right) \\
& \leq \frac{1}{p-2}\left(\frac{\bar c_{2,p-2}}{c_{p-2,2}}\right)^{p-2}\left(\|\tilde \x_{i^*}\|^{p-2} + \|\x_{i^*}\|^{p-2}\right).
\end{align*}
Since the upper bound 
\[
P_1(\|\x_{i^*}\|, \|\tilde \x_{i^*}\|) = \frac{1}{p-2}\left(\frac{\bar c_{2,4}}{c_{4,2}}\right)^p\left(\frac{\bar c_{2,p-2}}{c_{p-2,2}}\right)^{p-2}\left(\|\tilde \x_{i^*}\|^{p-2} + \|\x_{i^*}\|^{p-2}\right)
\]
is independent of $j$, it then follows that
\begin{align*}
\|\Theta_i - \tilde \Theta_i\| &= \sum_{j=1}^{p-1} \|\bigodot_{j=1}^{p-1}A_{k_j} - \bigodot_{j=1}^{p-1}\tilde A_{k_j}\| \leq P_1(\|\x_{i^*}\|,\|\tilde\x_{i^*}\|) \sum_{j=1}^{p-1}\|\x_{k_j}-\tilde \x_{k_j}\| \\
& \leq P_1(\|\x_{i^*}\|,\|\tilde\x_{i^*}\|) \left(\frac{\bar c_{2,1}}{c_{1,2}}\right) \|\x_{i^*}-\tilde \x_{i^*}\|,
\end{align*}
which establishes \eqref{eq:difference_of_thetas} with $P_\Theta(\|\x_{i^*}\|,\|\tilde\x_{i^*}\|) = \left(\frac{\bar c_{2,1}}{c_{1,2}}\right) P_1(\|\x_{i^*}\|,\|\tilde\x_{i^*}\|)$.
\end{proof}

In the special case when $\tilde \Theta_i=0$, Inequality \eqref{eq:difference_of_thetas} reduces to
\begin{equation}\label{eq:theta_bound}
\|\Theta_i\| \leq C_\Theta\|\x_{i^*}\|^{p-1}, \qquad \text{with $C_\Theta$ independent of $\x$}.
\end{equation} 

\begin{corollary}\label{cor:difference_of_theta_t_theta}
For $\Theta_i$ and $\tilde \Theta_i$ defined in \eqref{eq:theta_definition},  
\begin{equation}\label{eq:difference_of_theta_t_thetas}
\|\Theta_i^T\Theta_i - \tilde \Theta_i^T \tilde \Theta_i\| \leq P_{\Theta^T\Theta}(\|\x_{i^*}\|,\|\tilde \x_{i^*}\|)  \|\x_{i^*}-\tilde \x_{i^*}\|, \qquad \text{where}
\end{equation}
\begin{equation}
P_{\Theta^T\Theta}(\|\x_{i^*}\|,\|\tilde \x_{i^*}\|) = C_{\Theta^T\Theta} \max\{\|\x_{i^*}\|^{p-1}\!,\|\tilde \x_{i^*}\|^{p-1}\}\left(\|\x_{i^*}\|^{p-2}\! +\|\tilde \x_{i^*}\|^{p-2}\right),
\end{equation}
for some constant $0 \leq C_{\Theta^T\Theta}<\infty$ independent of $\x$ and $\tilde \x$.
\end{corollary}
\begin{proof}
By the equivalence of the induced Euclidean and the Frobenius norms, 
\begin{align*}
\|\Theta_i^T\Theta_i - \tilde \Theta_i^T \tilde \Theta_i\| & \leq c_{F,2}\|\Theta_i^T\Theta_i - \tilde \Theta_i^T \tilde \Theta_i\|_2 = c_{F,2}\|\Theta_i^T(\Theta_i-\tilde \Theta_i) + (\Theta_i-\tilde \Theta_i)^T \tilde \Theta_i\|_2\\
&\leq \frac{c_{F,2}}{(c_{2,F})^2} \left(\|\Theta_i^T\| \|\Theta_i-\tilde \Theta_i\| + \|\Theta_i-\tilde \Theta_i\| \|\tilde \Theta_i\|\right)\\
&\leq \frac{2 c_{F,2}}{(c_{2,F})^2} \max\{\|\Theta_i\|,\|\tilde \Theta_i\|\} \|\Theta_i-\tilde \Theta_i\| \\
&\leq \frac{2 c_{F,2}}{(c_{2,F})^2} \max\{\|\x_{i^*}\|^{p-1},\|\tilde \x_{i^*}\|^{p-1}\} P_\Theta(\|\x_{i^*}\|,\|\tilde \x_{i^*}\|)\|\x_{i^*}-\tilde \x_{i^*}\|.
\end{align*}
\end{proof}
Letting $\tilde \Theta_i = 0$ in \eqref{eq:difference_of_theta_t_thetas}, yields the bound
\begin{equation}\label{eq:theta_t_theta_bounded}
\|\Theta_i^T\Theta_i\| \leq C_{\Theta^T\Theta} \|\x_{i^*}\|^{2(p-1)}. 
\end{equation}

As a first consequence of \cref{lem:difference_of_thetas} and \cref{cor:difference_of_theta_t_theta}, we can obtain an explicit form for the component-wise Lipschitz constant of $\nabla_{\x_i}f(\x)$.
\begin{lemma}\label{lem:compgrad_lipschitz}
For any $\x, \tilde \x \in \R^{rn}$, there exists $P_{L}=P_{L}(\|\x\|,\|\tilde \x_{i^*}\|,\|\X\|)$, given by \eqref{eq:P_L}, so that 
\begin{equation}\label{eq:compgrad_lipschitz}
\|\nabla_{\x_i}f(\x;\X)-\nabla_{\x_i} f(\tilde \x;\X)\| \leq P_{L}(\|\x\|, \|\tilde \x_{i^*}\|, \|\X\|) \|\x - \tilde \x\|.
\end{equation}
\end{lemma}
\begin{proof}
Let $\Theta_i$ and $\tilde \Theta_i$ be constructed from $\x_{i^*}$ and $\tilde \x_{i^*}$ respectively via Equation \cref{eq:theta_definition}. Using \cref{eq:componentwise_stochastic_gradient}, the difference in sampled component-wise gradients is given by
\begin{align*}
&\nabla_{\x_i}f(\x;\X)-\nabla_{\bm x_i} f(\tilde \x;\X) \\
= & \left((\Theta_i^T-\tilde \Theta_i^T)\otimes I_{n_i}\right)\mathrm{vec}(\X_{(i)})  + (\Theta_i^T\Theta_i + \lambda I_r)\otimes I_{n_i}\x_i - (\tilde \Theta_i^T\tilde \Theta_i + \lambda I_r)\otimes I_{n_i}\tilde \x_i \\
= & \left((\Theta_i^T-\tilde \Theta_i^T)\otimes I_{n_i}\right)\mathrm{vec}(\X_{(i)})  + (\tilde \Theta_i^T\tilde \Theta_i + \lambda I_r)\otimes I_{n_i}(\x_i-\tilde \x_i) \\ 
& + (\Theta_i^T\Theta_i - \tilde \Theta_i^T\tilde \Theta_i)\otimes I_{n_i}\bm x_i.
\end{align*}

Since the singular values of a Kronecker product are formed from products of singular values of the constituent matrices,  the matrix norm $\|B\otimes I\|_2 = \|B\|_2 \leq \|B\|$ for any matrix $B$. In light of inequalities \eqref{eq:difference_of_thetas} and \eqref{eq:difference_of_theta_t_thetas}, we therefore have 
\begin{align*}
\;&\left\|\left((\Theta_i^T-\tilde \Theta_i^T)\otimes I_{n_i}\right)\mathrm{vec}(\X_{(i)})+(\Theta_i^T\Theta_i - \tilde \Theta_i^T\tilde \Theta_i)\otimes I_{n_i}\bm x_i\right\| \\
\leq\; &\|\Theta_i-\tilde \Theta_i\|_2 \|\X\| + \|\Theta_i^T\Theta_i - \tilde \Theta_i^T\tilde \Theta_i\|_2 \|\bm x_i\| \\
\leq \;& \big( P_{\Theta}(\|\x_{i^*}\|,\|\tilde \x_{i^*}\|)\|\X\| +  P_{\Theta^T\Theta}(\|\x_{i^*}\|,\|\tilde \x_{i^*}\|)\|\x_i\| \big) \|\x_{i^*}-\tilde \x_{i^*}\|.
\end{align*}

Similarly, Inequality \eqref{eq:theta_t_theta_bounded} yields
\begin{align*}
\;&\|(\tilde \Theta_i^T\tilde \Theta_i + \lambda I_r)\otimes I_{n_i}(\x_i-\tilde \x_i)\|\leq \|\tilde \Theta_i^T\tilde \Theta_i + \lambda I_r\|_2 \|\x_i-\tilde \x_i\|\\
\leq\; &\big(\|\tilde \Theta_i^T\tilde \Theta_i\| + \lambda \big) \|\x_i-\tilde \x_i\| \leq \left(C_{\Theta^T\Theta}\|\tilde \x_{i^*}\|^{2(p-1)}+\lambda\right) \|\x_i-\tilde \x_i\| .
\end{align*}
Hence, by norm equivalence \eqref{eq:norm_equivalence} with $s=1$ and $t=F$, the bound \eqref{eq:compgrad_lipschitz} follows with
\begin{align}\label{eq:P_L}
&P_{L}(\|\x\|, \|\tilde \x_{i^*}\|, \|\X\|) = \left(\frac{\bar c_{F1}}{c_{1F}}\right)\max \left\{P_{L_{i^*}}(\|\x\|,\|\tilde \x_{i^*}\|,\|\X\|),P_{L_i}(\|\tilde \x_{i^*}\|)\right\}, 
\\\nonumber
&\hspace{-3em} \text{with} \quad 
P_{L_i}(\|\tilde \x_{i^*}\|) = C_{\Theta^T\Theta}\|\tilde \x_{i^*}\|^{2(p-1)}+\lambda
\\ \nonumber
&\hspace{-3em} \text{and}\quad
P_{L_{i^*}}(\|\x\|,\|\tilde \x_{i^*}\|,\|\X\|) =  P_{\Theta}(\|\x_{i^*}\|,\|\tilde \x_{i^*}\|)\|\X\| \!+\!  P_{\Theta^T\Theta}(\|\x_{i^*}\|,\|\tilde \x_{i^*}\|)\|\x_i\|.\hspace{-1em}
\end{align}
\end{proof}
The following is a simplified upper bound in the special case when $\x$ and $\tilde \x$ only differ in their $i$-th block components
\begin{corollary}\label{cor:compgrad_comp_lipschitz}
For any $\x=(\x_1,\hdots,\x_p)=(\x_i,\x_{i^*}) \in \R^{rn}$, we have
\begin{equation}
\|\nabla_{\x_i} f(\x_i,\x_{i^*};\X)-\nabla_{\x_i}f(\tilde \x_i,\x_{i^*};\X)\| \leq P_{L_i}(\|\tilde \x_{i^*}\|) \|\x_i-\tilde \x_i\|,
\end{equation}
\end{corollary}
\begin{proof}
Since $\x_{i^*}=\tilde \x_{i^*}$, estimate \eqref{eq:compgrad_lipschitz} collapses to
\[
\|(\Theta_i^T \Theta_i + \lambda I_r)\otimes I_{n_i}(\x_i-\tilde \x_i)\| \leq P_{L_i}(\|\x_{i^*}\|)\|\x_i-\tilde \x_i\|.
\]
\end{proof}

We may also bound the norm of the component-wise gradient in terms of the norm of $\x = (\x_i,\x_{i^*})$ and the data $\X$, as shown in the following corollary. 

\begin{corollary}\label{cor:bnd_gradient_norm}
For any $\x = (\x_i,\x_{i^*}) \in \R^{rn}$, 
\begin{equation}\label{eq:bnd_gradient_norm}
\|\nabla_{\x_i} f(\x;\X)\| \leq P_{L}(\|\x\|,0,\|\X\|)\|\x_i\|,
\end{equation}
where $P_L$ is given by \eqref{eq:P_L}.
\end{corollary}
\begin{proof}
To get inequality \eqref{eq:bnd_gradient_norm} set $\tilde \x=\bm 0$ in \eqref{eq:compgrad_lipschitz} and note that ${\nabla_{\x_i}f(\bm 0;\X) = \bm 0}$.
\end{proof}
\Cref{cor:difference_of_theta_t_theta} also implies the following bounds on the component-wise Hessian.
\begin{corollary}\label{cor:hessian_posdef}
For any $\x = (\x_i,\x_{i^*}) \in \R^{rn}$ and ${\bm v}_i \in \R^{rn_i}$, 
\begin{equation}\label{eq:hessian_posdef}
\lambda \|\bm{v}_i\|^2 \leq \bm{v}_i^T \nabla_{\x_i}^2 f(\x) \bm{v}_i \leq P_{L_i}(\|\x_{i^*}\|) \|\bm{v}_i\|^2.
\end{equation}

\end{corollary}
\begin{proof}
This follows directly from the fact that
\[
\lambda \leq \|\nabla_{\x_i}^2 f(\x)\|_2 \leq \|\Theta_i^T \Theta_i\| + \lambda \leq P_{L_i}(\| \x_{i^*}\|).
 \]
\end{proof}

Finally, we establish the local Lipschitz continuity of the Newton step mapping $\x\mapsto(\nabla_{\x_i}^2 f(\x))^{-1} \nabla_{\x_i}f(\x;\X)$. 

\begin{lemma}\label{lem:newton_step_lipschitz}
For any $\x, \bm y \in \R^{rn}$, we have
\begin{multline}\label{eq:newton_step_lipschitz}
\|(\nabla_{\x_i}^2 f(\x))^{-1} \nabla_{\x_i}f(\x;\X)-(\nabla_{\x_i}^2 f(\tilde \x))^{-1} \nabla_{\x_i}f(\tilde \x;\X)\| \\
\leq P_N(\|\x\|,\|\tilde \x_{i^*}\|,\|\X\|) \|\x-\tilde \x\|,
\end{multline}
where $P_N$ is given by \eqref{eq:P_N}.

\end{lemma}
\begin{proof}
Let $\x,\tilde \x\in \R^{rn}$. By adding and subtracting a cross-term, we obtain
\begin{align*}
&\phantom{+}\|(\nabla_{\x_i}^2 f(\x))^{-1} \nabla_{\x_i}f(\x;\X)-(\nabla_{\x_i}^2 f(\tilde \x_i))^{-1} \nabla_{\x_i}f(\tilde \x;\X)\|\\[0.5em]
\leq & \phantom{+} \|(\nabla_{\x_i}^2 f(\x))^{-1} -(\nabla_{\tilde \x_i}^2 f(\tilde \x))^{-1} \|_2 \|\nabla_{\x_i}f(\x;\X)\| \\[0.5em]
&+ \|(\nabla_{\x_i}^2 f(\tilde \x))^{-1}\|_2 \| \nabla_{\x_i}f(\x;\X)-\nabla_{\x_i}f(\tilde \x;\X)\|.
\end{align*}
By the second resolvent identity (see e.g. Theorem 4.8.2. in \cite{Hille1974}), we have 
\begin{align*}
(\nabla_{\x_i}^2 f(\x))^{-1} -(\nabla_{\tilde \x_i}^2 f(\tilde \x))^{-1} &= (\Theta_i^T\Theta_i + \lambda I_r)^{-1} - (\tilde \Theta_i^T \tilde \Theta_i + \lambda I_r)^{-1} \\
&= (\Theta_i^T\Theta_i + \lambda I_r)^{-1}(\tilde \Theta_i\tilde \Theta_i - \Theta_i^T\Theta_i)(\tilde \Theta_i^T \tilde \Theta_i + \lambda I_r)^{-1},
\end{align*}
so that, by virtue of \eqref{eq:difference_of_theta_t_thetas}, \eqref{eq:bnd_gradient_norm}, and \eqref{eq:hessian_posdef},
\begin{align*}
& \|(\nabla_{\x_i}^2 f(\x))^{-1} -(\nabla_{\tilde \x_i}^2 f(\tilde \x))^{-1} \|_2 \|\nabla_{\x_i}f(\x;\X)\| \\
\leq &\frac{1}{\lambda^2} \|\x_i\|P_L(\|\x\|,0,\|\X\|)P_{\Theta^T\Theta}(\|\x_{i^*}\|,\|\tilde \x_{i^*}\|)\|\x_{i^*}-\tilde \x_{i^*}\|.
\end{align*}
Moreover, \eqref{eq:compgrad_lipschitz} and \eqref{eq:hessian_posdef} imply
\begin{align*}
\|(\nabla_{\x_i}^2 f(\tilde \x))^{-1}\|_2 \| \nabla_{\x_i}f(\x;\X)-\nabla_{\x_i}f(\tilde \x;\X)\| \leq \frac{1}{\lambda}P_{L}(\|\x\|,\|\tilde \x_{i^*}\|, \|\X\|)\|\x_i-\tilde \x_i\|.
\end{align*}
Combining these estimates gives
\[
\|(\nabla_{\x_i}^2 f(\x))^{-1} \nabla_{\x_i}f(\x;\X)-(\nabla_{\x_i}^2 f(\tilde \x_i))^{-1} \nabla_{\x_i}f(\tilde \x;\X)\|\leq P_{N} \|\x-\tilde \x\|,
\quad \text{with}\]
\begin{equation}\label{eq:P_N}
\begin{split}
P_N(\|\x\|,\|\tilde \x_{i^*}\|,\|\X\|) = \left(\frac{\bar c_{2,1}}{c_{1,2}}\right)\max\left\{ P_2, P_3\right\},
\end{split}
\end{equation}
where $\bar c_{2,1}$ and $c_{1,2}$ are appropriate norm equivalence constants,
\begin{align*}
P_2(\|\x\|,\|\tilde \x_{i^*}\|,\|\X\|) &= \frac{1}{\lambda^2} \|\x_i\|P_L(\|\x\|,0,\|\X\|)P_{\Theta^T\Theta}(\|\x_{i^*}\|,\|\tilde \x_{i^*}\|), \qquad \text{and} \\
P_3(\|\x\|,\|\tilde \x_{i^*}\|,\|\X\|) &= \frac{1}{\lambda}P_{L}(\|\x\|,\|\tilde \x_{i^*}\|, \|\X\|).
 \end{align*}
\end{proof}

\subsection{Boundedness of the Data}\label{subsection:boundedness}

In this section we bound the norm of the component iterates $\x_i^{k}$ in terms of the norm of the initial guess and maximum value of the sequence of sample averages of the norms of $\X$.

\begin{lemma}\label{lem:iterates_bnded}
The iterates $\x^{k,i}$ generated by \cref{alg:stochastic_als} satisfy
\[
\E\left[\|\x_i^{k+1}\|\right]\leq \max \left\{\|\x_i^1\|,R(\bX^{k}),\hdots,R(\bX^1)\right\}
\]
for each $k=1,2\hdots$, and $i=1,\hdots,p$, where $R(\bX^k)$ is given by \eqref{eq:bnd_minimizer}.
\end{lemma}
\begin{proof}
Note that the component-wise minimizer $\hat \x_i^{k+1}$ given in \eqref{eq:componentwise_minimizer} satisfies 
\begin{align*}
& \|\hat \x_{i}^{k+1}\|^2 + \sum_{j=1}^{i-1}\|\x_{j}^{k+1}\|^2 + \sum_{j=i+1}^p \|\x_{j}^k\|^2  \\
\leq & \frac{2}{\lambda} \tilde f(\x_1^{k+1},\hdots,\x_{i-1}^{k+1},\hat \x_{i}^{k+1},\x_{i+1}^k,\hdots,\x_{p}^k;\bX^k) & \text{(by \eqref{eq:bound_normx_by_f})}\\
\leq & \frac{2}{\lambda} \tilde f(\x_1^{k+1},\hdots,\x_{i-1}^{k+1},{\bm 0},\x_{i+1}^k,\hdots,\x_{p}^k;\bX^k) & \text{(by optimality)}\\
= & \frac{1}{\lambda}\left(\frac{1}{m_k}\sum_{l=1}^{m_k}\|\X^{k,l}\|^2\right) + \sum_{j=1}^{i-1}\|\x_{j}^{k+1}\|^2 + \sum_{j=i+1}^p \|\x_{j}^k\|^2,
\end{align*}
so that
\begin{equation}\label{eq:bnd_minimizer}
\|\hat \x_{i}^{k+1}\| \leq \sqrt{\frac{1}{\lambda} \left(\frac{1}{m_k}\sum_{l=1}^{m_k}\|\X^{k,l}\|^2\right)} =: R(\bX^k).
\end{equation}
To bound the iterates $\x_i^{k+1}$  note that Equation \eqref{eq:x_update_sals} can be rewritten as the convex combination
\[
\x_i^{k+1} = \alpha^{k,i}\x_i^{k} + (1-\alpha^{k,i})\hat \x_i^{k+1},
\]
which, by virtue of the stepsize bounds $0\leq \alpha^{k,i}\leq 2$, implies
\begin{align*}
\|\x_i^{k+1}\| &\leq |\alpha^{k,i}| \|\x^{k,i}\| + |1-a^{k,i}| \|\hat \x^{k,i}\| \leq \max\{ \|\x_i^k\|, R(\bX^k)\} (|\alpha^{k,i}|+|1-\alpha^{k,i}|)\\
&\leq \max\{ \|\x_i^k\|, R(\bX^k)\} \leq \max\left\{\|\x_i^{k-1}\|,R(\bX^k),R(\bX^{k-1})\right\} \leq \hdots \\
&\leq \max \left\{\|\x_i^1\|,R(\bX^k),R(\bX^{k-1}),\hdots,R(\bX^1)\right\}.
\end{align*}
\end{proof}

Since the regularity estimates derived above all involve powers of $\|\x\|$ and of $\|\bX^k\|$, \cref{lem:iterates_bnded} suggests that a bound on the data $\X$ is sufficient to guarantee the regularity of the cost functional, gradient, and Newton steps that are necessary to show convergence. In the following we assume such a bound and pursue its consequences.

\begin{assumption}[Bounded data]\label{ass:tensor_bounded}
There is a constant $0 < M < \infty$ so that 
\begin{equation}
\|\X\| \leq M, \ \ a.s. \text{ on }  \Omega.
\end{equation}
\end{assumption}

\begin{remark}
This assumption might conceivably be weakened to one pertaining to the statistical distribution of the maxima $R(\bX^1),R(\bX^2),\hdots, R(\bX^k)$. Specifically, letting $r_k=\max_{l=1,\hdots,k} R(\bX^l)$, it can be shown under appropriate conditions on the density of $R(\X^k)$, that $r_k$ converges in distribution to a random variable with known extreme value density. The analysis below will hold if it can be guaranteed that the limiting distribution has bounded moments of sufficiently high order. This possibility will be pursued in future work. 
\end{remark}

An immediate consequence of \cref{ass:tensor_bounded} is the existence of a uniform bound on the radius $R(\bX)$ and hence on the iterates $\x^{k,i}$.
\begin{corollary}\label{cor:x_bounds} Given \cref{ass:tensor_bounded}, there exist finite, non-negative constants $M_R, M_{x_i}$, and $M_x$ independent of $\x^k$ and of $\X^k$, so that for all $k=1,2,\hdots$,
\begin{align}
R(\bX^k) &\leq M_{R{\phantom{ \x_i}}} \qquad \text{a.s.\! on } \Omega \label{eq:radius_bnd}\\
\|\x_i^k\| &\leq M_{\x_i\phantom{R}} \qquad \text{a.s.\! on } \Omega \label{eq:xi_bnd}\\ 
\|\x^k\| &\leq M_{\x_{\phantom i}\phantom{R}} \qquad \text{a.s.\! on } \Omega \label{eq:x_bnd}
\end{align}
\end{corollary}
\begin{proof}
By \eqref{eq:bnd_minimizer} and \cref{ass:tensor_bounded},
\begin{align*}
R(\bX^k) =\sqrt{\frac{1}{\lambda} \left(\frac{1}{m_k}\sum_{l=1}^{m_k}\|\X^{k,l}\|^2\right)} \leq \frac{M}{\sqrt{\lambda}} =: M_R.
\end{align*}
\Cref{lem:iterates_bnded} then implies
\[
\|\x_i^k\| \leq \max \{\|\x_i^1\|,M_R\} =: M_{\x_i}
\]
and hence
\[
\|\x^k\| = \sqrt{\sum_{i=1}^p \|\x_i^k\|^2} \leq \sqrt{p}M_{\x_i} =: M_{\x}.
 \]
\end{proof}

\subsection{Convergence}\label{subsection:convergence_proof}
We now consider the difference $\bm \delta^k = \tilde{\bm g}^{k}-\bm g^k$ between the sampled and expected search directions. For the standard stochastic gradient descent algorithm, this stochastic quantity vanishes in expectation, given past information $\mathcal{F}^{k-1}$, i.e. $\E\left[ {\bm \delta^{k}}\vert \mathcal F^{k-1}\right]=0$, since $\tilde \g^k$ is an unbiased estimator of $\g^k$. For the stochastic alternating least squares method, this is no longer the case. \Cref{lem:bound_step_conditional} however uses the regularity of the gradient and the Hessian to establish an upper bound that decreases on the order $O(\frac{1}{k})$ as $k\rightarrow \infty$. 

\begin{lemma}\label{lem:bound_step_conditional}

There is a constant $M_N\geq 0$ such that for every $k=1,2,...$ and $i=1,...,p$,
\begin{equation}\label{eq:bound_step_conditional}
\left\|\E\left[ (H^{k,i})^{-1}{\bm \delta^{k,i}}\vert \mathcal F^{k-1}\right]\right\| \leq \frac{M_N}{k},
\end{equation}
where $M_N = 2\cmax P_N(M_{\x},0,M_{R})P_N(M_{\x},M_{\x_{i^*}},M_{R}) M_{\x}$.
\end{lemma}

\begin{proof} 

Recall that the current iterate $\x^{k,i}$ is statistically dependent on sampled tensors $\bX^1,\hdots,\bX^{k-1}$, while its first $i$ components also depend on $\bX^k$. For the purpose of computing $\E\left[ (H^{k,i})^{-1}{\bm \delta^{k,i}}\vert \mathcal F^{k-1}\right]$, we suppose that $\bX^1,\hdots, \bX^{k-1}$ are known and write $\x^{k,i-1} = \x^{k,i-1}(\bX^k)=(\x_1^{k+1}(\bX^k),\hdots,\x_{i-1}^{k+1}(\bX^k),\x_i^k,\hdots,\x_p^k)$ to emphasize its dependence on $\bX^k$. Thus
\begin{align*}
& \E\left[ (H^{k,i})^{-1}{\bm \delta^{k,i}}\vert \mathcal F^{k-1}\right] = \\
&\int_{\Omega^{m_k}} \left(\nabla_{\x_i}^2 f(\x^{k,i-1}(\bX^k))\right)^{-1} \left(\nabla_{\x_i}\tilde f(\x^{k,i-1}(\bX^k);\bX^k)-\nabla_{\x_i}F(\x^{k,i-1}(\bX^k))\right) d\mu_{\bX^k}.
\end{align*}
By definition, and since $\nabla_{\x_i} \tilde f$ is an unbiased estimator of $\nabla_{\x_i} F$, we have
\begin{align*}
\nabla_{\x_i} F(\x^{k,i-1}(\bX^k)) &= \int_{\Omega} \nabla_{x_i}f(\x^{k,i-1}(\bX^k);\X)d\mu_\X \\
&= \int_{\Omega^{m_k}} \nabla_{x_i}\tilde f(\x^{k,i-1}(\bX^k);\bX)d\mu_{\bX}.
\end{align*}
Moreover, since $\bX$ and $\bX^k$ are identically distributed, 
\begin{align*}
& \int_{\Omega^{m_k}} \left(\nabla_{\x_i}^2 f(\x^{k,i-1}(\bX^k))\right)^{-1} \nabla_{\x_i}F(\x^{k,i-1}(\bX^k)) d\mu_{\bX^k} \\
= &\int_{\Omega^{m_k}}\int_{\Omega^{m_k}} \left(\nabla_{\x_i}^2 f(\x^{k,i-1}(\bX^k))\right)^{-1} \nabla_{\x_i}\tilde f(\x^{k,i-1}(\bX^k);\bX) d\mu_{\bX} d\mu_{\bX^k}\\
= &\int_{\Omega^{m_k}}\int_{\Omega^{m_k}} \left(\nabla_{\x_i}^2 f(\x^{k,i-1}(\bX))\right)^{-1} \nabla_{\x_i}\tilde f(\x^{k,i-1}(\bX);\bX^k) d\mu_{\bX} d\mu_{\bX^k}.
\end{align*}
Therefore
\begin{multline}\label{eq:cond_expect_step}
\E\left[(H^{k,i})^{-1}{\bm \delta^{k,i}}\vert \mathcal F^{k-1}\right] \\
= \int_{\Omega^{m_k}}\int_{\Omega^{m_k}} \left(\nabla_{\x_i}^2 f(\x^{k,i-1}(\bX^k))\right)^{-1} \nabla_{\x_i}\tilde f(\x^{k,i-1}(\bX^k);\bX^k) d\mu_{\bX}d\mu_{\bX^k}\\
-\int_{\Omega^{m_k}}\int_{\Omega^{m_k}} \left(\nabla_{\x_i}^2 f(\x^{k,i-1}(\bX))\right)^{-1} \nabla_{\x_i}\tilde f(\x^{k,i-1}(\bX);\bX^k) d\mu_{\bX}d\mu_{\bX^k}.
\end{multline}
In the special case  $i=1$, the iterate $\x^{k,0}=\x^{k-1}$, and hence $H^{k,1}$, does not depend on $\bX^k$. Since $\nabla_{\x_1}\tilde f(\x^{k-1};\bX^k)$ is an unbiased estimator of $\nabla_{\x_1}F(\x^{k-1})$, we  have
\[
\E\left[(H^{k,1})^{-1}{\bm \delta^{k,1}}\vert \mathcal F^{k-1}\right] = 0.
\] 
We now consider the case $i=2,...,p$. Using the Lipschitz continuity of the mapping $\x \mapsto\left(\nabla_{\x_i}^2 f(\x)\right)^{-1} \nabla_{\x_i} f(\x;\bX^k)$ (\cref{lem:newton_step_lipschitz}), the bounds in \cref{cor:x_bounds}, and Jensen's inequality, we obtain
\begin{multline*}
\left\|\E\left[(H^{k,i})^{-1}{\bm \delta^{k,i}}\vert \mathcal F^{k-1}\right]\right\| \\
\leq P_N(M_\x,M_{\x_{i^*}},M) \int_{\Omega^{m_k}}\int_{\Omega^{m_k}}  \|\x^{k,i-1}(\bX^k)-\x^{k,i-1}(\bX)\| d\mu_{\bX}d\mu_{\bX^k},
\end{multline*}
the integrand of which can be bounded by
\begin{multline*}
\|\x^{k,i-1}(\bX^k)-\x^{k,i-1}(\bX)\| \leq \sum_{j=1}^{i-1} \|\x_j^{k+1}(\bX^k)-\x_j^{k+1}(\bX)\|\\
 =\sum_{j=1}^{i-1} \alpha^{k,j} \|( H^{k,j}(\bX^k))^{-1}\tilde \g^{k,j}(\bX^k)-( H^{k,j}(\bX))^{-1}\tilde \g^{k,j}(\bX)\|\\
 \leq \frac{\cmax}{k} \left(\|( H^{k,j}(\bX^k))^{-1}\tilde \g^{k,j}(\bX^k)\|+ \|( H^{k,j}(\bX))^{-1}\tilde \g^{k,j}(\bX)\|\right).
\end{multline*}
The result now follows from taking expectations and using \eqref{eq:newton_step_lipschitz} with $\tilde \x=\bm 0$, in conjunction with \eqref{eq:radius_bnd}, and \eqref{eq:x_bnd}.
\end{proof}

\begin{lemma}\label{lem:bound_measurable_times_newton}
If ${\bm h}$ is $\mathcal F^{k-1}$-measurable and $\E[\|\bm{h}\|]<\infty$ then 
\begin{equation}
\E\left[\left\langle {\bm h}, ( H^{k,i})^{-1} {\bm \delta}^{k,i} \right \rangle \right] \leq \frac{M_N}{k} \E\left[\|\bm{h}\|\right].
\end{equation}

\end{lemma}

\begin{proof}
Using the law of total expectation, the $\mathcal F^{k-1}$-measurability of $\bm{h}$, and \cref{lem:bound_step_conditional}, we have
\begin{align*}
\E\left[\left\langle {\bm h}, ( H^{k,i})^{-1} {\bm \delta}^{k,i} \right \rangle\right] &= \E_{\bX^1,\hdots,\bX^{k-1}} \left[ \E\left[\left\langle {\bm h}, ( H^{k,i})^{-1} {\bm \delta}^{k,i} \right \rangle \middle| \mathcal F^{k-1}\right] \right]\\
&= \E_{\bX^1,\hdots,\bX^{k-1}} \left[ \left\langle \E\left[{\bm h}\middle| \mathcal F^{k-1}\right], \E\left[( H^{k,i})^{-1} {\bm \delta}^{k,i} \middle| \mathcal F^{k-1}\right] \right \rangle  \right]\\
& \leq  \E_{\bX^1,\hdots,\bX^{k-1}} \left[ \left\| \E\left[{\bm h}\right]\right\|\left\|\E\left[( H^{k,i})^{-1} {\bm \delta}^{k,i} \middle| \mathcal F^{k-1}\right] \right \|  \right]\\
&\leq  \E_{\bX^1,\hdots,\bX^{k-1}} \left[\frac{M_N}{k} \E_{\bX^{k}}[\|\bm{h}\|]\right] = \frac{M_N}{k}\E[\|\bm h\|].
 \end{align*}
\end{proof}

The main convergence theorem is based on the following lemma (for a proof, see e.g. Lemma A.5, \cite{Mairal2013})
\begin{lemma}\label{lem:pos_sequence}
Let $\{a_k\}_{k=1}^\infty$ and $\{b_k\}_{k=1}^\infty$ be any two nonnegative, real sequences so that (i) $ \sum_{k=1}^\infty a_k = \infty$, (ii) $\sum_{k=1}^\infty a_k b_k < \infty$, and (iii) there is a constant $K>0$ so that $|b_{k+1}-b_k|\leq Ka_k$ for $k\geq 1$. Then $\lim_{k\rightarrow \infty} b_k = 0$.
\end{lemma}

\begin{theorem}\label{thm:convergence}
Let ${\bm x}^{k,i}$ be the sequence generated by the Stochastic Alternating Least Squares (SALS) method outlined in \cref{alg:stochastic_als}. Then 
\[
\lim_{k\rightarrow \infty} \E\left[\|\nabla_{\bm x_i}F({\bm x}^{k,i})\|^2\right] = 0, \qquad i=1,...,p.
\]
\end{theorem}
\begin{proof}
We base the proof on \cref{lem:pos_sequence} with $a_k = \frac{1}{k}$ and $b_k = \|\bm g^{k,i}\|^2$. Clearly, Condition (i) in \cref{lem:pos_sequence} is satisfied.  To show that Condition (ii) holds, i.e. that $\sum_{k=1}^\infty \frac{1}{k}\E\left[\|\g^{k,i}\|^2\right]<\infty$ for $i=1,\hdots, p$, we use the component-wise Taylor expansion \eqref{eq:expected_cost_taylor} of the expected cost centered at the iterate $\x^{k,i-1}$ and the SALS update given in \eqref{eq:x_update_sals} to express the expected decrease as
\begin{align*}
& F({\bm x}^{k,i}) - F({\bm x}^{k,i-1}) \\
= & \nabla_{\bm x_i} F({\bm x}^{k,i})^T({\bm x}^{k,i+1} - {\bm x}^{k,i}) + \frac{1}{2}({\bm x}^{k,i+1} - {\bm x}^{k,i})^T \nabla_{{\bm x}_i}^2 F({\bm x}^{k,i})({\bm x}^{k,i+1} - {\bm x}^{k,i})\\
=& -\alpha^{k,i} ({\bm g}^{k,i})^T ( H^{k,i})^{-1}\tilde{\bm g}^{k,i} + \frac{1}{2}(\alpha^{k,i})^2 (( H^{k,i})^{-1}\tilde{\bm g}^{k,i})^T  H^{k,i}( H^{k,i})^{-1}\tilde {\bm g}^{k,i}\\
= & -\alpha^{k,i} {\bm g^{k,i}}^T ( H^{k,i})^{-1} \bm g^{k,i} - \alpha^{k,i} (\bm g^{k,i})^T ( H^{k,i})^{-1} \bm \delta^{k,i}+ \frac{1}{2}(\alpha^{k,i})^2 (\tilde{\bm{g}}^{k,i})^T (H^{k,i})^{-1} \tilde{\bm{g}}^{k,i}.
\end{align*}
Since, by \eqref{eq:xi_bnd} and Corollary \eqref{cor:hessian_posdef}, 
\[
\alpha^{k,i} (\bm g^{k,i})^T ( H^{k,i})^{-1} \bm g^{k,i} \geq \frac{\cmin}{k} \frac{1}{\lambda + M_{\nabla^2 f}} \|\bm g^{k,i}\|^2 ,
\]
where $M_{\nabla^2 f} = P_{L_i}(M_{\x_i^*})$, the above expression can be rearranged as
\begin{align}\label{eq:conv_bounds1to3}
\frac{1}{k}\|\g^{k,i}\|^2 &\leq \frac{\lambda + M_{\nabla^2 f}}{\cmin}\left(E_1^{k,i} + E_2^{k,i} + E_3^{k,i}\right), \quad\text{with} 
\\\nonumber
&E_1^{k,i} = F({\bm x}^{k,i-1}) - F({\bm x}^{k,i}), 
\\\nonumber
&E_2^{k,i} = -\alpha^{k,i} (\bm g^{k,i})^T (\tilde H^{k,i})^{-1} \bm \delta^{k,i}, \ \text{and}
\\\nonumber
&E_3^{k,i} = \frac{1}{k^2}\frac{\cmax}{2} (\tilde{\bm{g}}^{k,i})^T (H^{k,i})^{-1} \tilde{\bm{g}}^{k,i}.
\end{align}
Evidently, Condition (ii) of \cref{lem:pos_sequence} holds as long as $\sum_{k=1}^{\infty} \E\left[E_j^{k,i}\right] < \infty$ for $j=1,2,3$. For a fixed $K>0$, the first term $E_1^{k,i}$ generates a telescoping sum so that  
\begin{align}
&\sum_{k=1}^K \sum_{i=1}^p \E\left[E_1^{k,i}\right] = \sum_{k=1}^K \E\left[F({\bm x}^{k})\right] - \E\left[F({\bm x}^{k+1})\right] 
\nonumber \\
= \;&\E\left[F(\x^{1})\right] - \E\left[F(\x^{K+1})\right] \leq \E\left[F(\x^{1})\right] < \infty \qquad \forall K>0, \label{eq:conv_bound01}
\end{align}
since $F(\x^{K+1})\geq 0$. Consider 
\begin{align*}
E_2^{k,i} &= - \alpha^{k,i} (\bm g^{k,i})^T ( H^{k,i})^{-1} \bm \delta^{k,i} \\
&=  - \alpha^{k,i} ({\bm g}^{k,i}-\nabla_{{\bm x}_i}F({\bm x}^k))^T ( H^{k,i})^{-1}{\bm \delta}^{k,i} + \alpha^{k,i}\nabla_{{\bm x}_i}F({\bm x}^k)^T (H^{k,i})^{-1}{\bm \delta}^{k,i}.
\end{align*}
Since $\nabla_{{\bm x}_i}F({\bm x}^k)$ is $\mathcal F^{k-1}$-measurable, \cref{lem:bound_measurable_times_newton} implies that the expectation of the second term can be bounded as follows
\[
\E\left[\alpha^{k,i}\nabla_{{\bm x}_i}F({\bm x}^k)^T (H^{k,i})^{-1}{\bm \delta}^{k,i}\right] \leq \frac{M_{\nabla f} M_N}{k^2},
\]
where $M_{\nabla f} = P_{L}(M_{\x},0,M)M_{\x_i}$ by \cref{cor:bnd_gradient_norm,cor:x_bounds}. The first term satisfies
\begin{align*}
& - \alpha^{k,i} (\nabla_{{\bm x}_i}F({\bm x}^{k,i})-\nabla_{{\bm x}_i}F({\bm x}^k))^T ( H^{k,i})^{-1}{\bm \delta}^{k,i} \\
\leq \;& \alpha^{k,i} \|\nabla_{{\bm x}_i}F({\bm x}^{k,i})-\nabla_{{\bm x}_i}F({\bm x}^k)\| \|( H^{k,i})^{-1}{\bm \delta}^{k,i}\|\\
\leq \;&\alpha^{k,i} M_N M_L\|{\bm x}^{k,i}-{\bm x}^k\| \\
=\;& M_N M_L \alpha^{k,i} \sqrt{\sum_{j=1}^i (\alpha^{k,j})^2\|(H^{k,j})^{-1}\tilde {\bm g}^{k,j}\|^2} \\
\leq\;& p (M_N)^2 M_L \left(\frac{\cmax}{k}\right)^2,
\end{align*}
where $M_L = P_L(M_\x,M_{\x_{i^*}},M)$ by \cref{lem:compgrad_lipschitz,cor:x_bounds}.
Combining these two bounds yields
\begin{equation}\label{eq:conv_bound02}
 \sum_{k=1}^\infty \sum_{i=1}^p \E\left[E_2^{k,i}\right] \leq \left(p^2(\cmax M_N)^2 M_L  + p M_{\nabla f} M_N \right)\sum_{k=1}^\infty \frac{1}{k^2}<\infty.
\end{equation}
Finally, we use \cref{cor:bnd_gradient_norm,cor:hessian_posdef,cor:x_bounds} to bound
\begin{align*}
\E\left[(\tilde{\bm{g}}^{k,i})^T (H^{k,i})^{-1} \tilde{\bm{g}}^{k,i}\right] \leq & \frac{1}{\lambda}\E\left[\|\tilde \g^{k,i}\|^2 \right]\leq \frac{(M_{\nabla f})^2}{\lambda}, 
\end{align*}
so that
\begin{equation}\label{eq:conv_bound03}
\sum_{k=1}^\infty \sum_{i=1}^p\E\left[E_3^{k,i}\right] \leq \frac{p\cmax (M_{\nabla f})^2}{2\lambda} \sum_{k=1}^\infty \frac{1}{k^2} <\infty.  
\end{equation} 
By virtue of Inequality \eqref{eq:conv_bounds1to3}, the upper bounds \eqref{eq:conv_bound01}, \eqref{eq:conv_bound02}, and \eqref{eq:conv_bound03} now imply that Condition (ii) of \cref{lem:pos_sequence} holds.  It now remains to show that Condition (iii) of \cref{lem:pos_sequence} holds, i.e. that $\E\left[\|\g^{k+1,i}\|^2\right] - \E\left[\|\g^{k,i}\|^2\right] = O(1/k)$ as $k\rightarrow \infty$ for all $i=1,\hdots,p$. By the reverse triangle inequality and by the Lipshitz continuity and boundedness of the gradient, 
\begin{align*}
& \|\g^{k+1,i}\|^2 - \|\g^{k,i}\|^2 \leq (\|\g^{k+1,i}\| + \|\g^{k,i}\|)(\|\g^{k+1,i} - \g^{k,i}\|) \\
\leq\;& 2 M_{\nabla f} M_{L} \|\x^{k+1,i}-\x^{k,i}\| \\
=\;& 2 M_{\nabla f} M_{L} \sqrt{\sum_{j=1}^{i} \|\x_j^{k+2}-\x_j^{k+1}\|^2 + \sum_{j=i+1}^p \|\x_j^{k+1}-\x_j^{k}\|^2}\\
=\;& 2 M_{\nabla f} M_{L}\sqrt{\sum_{j=1}^{i} \left(\alpha^{k+1,j}\right)^2 \|(H^{k+1,j})^{-1}\tilde \g^{k+1,j}\|^2 + \sum_{j=i+1}^p \left(\alpha^{k,j}\right)^2\|(H^{k,j})^{-1}\tilde \g^{k,j}\|^2}\\
\leq\;& 2 M_{\nabla f} M_{L} \frac{\cmax}{k} \sqrt{\sum_{j=1}^{i} \|(H^{k+1,j})^{-1}\tilde \g^{k+1,j}\|^2 + \sum_{j=i+1}^p \|(H^{k,j})^{-1}\tilde \g^{k,j}\|^2}\\
\leq\;& (2p \cmax M_{\nabla f} M_{L} M_N) \frac{1}{k}.
 \end{align*}
\end{proof}

\section{Numerical Experiments}\label{section:numerical}

In this section we detail the implementation of the SALS algorithm, discuss its computational complexity and perform a variety of numerical experiments to explore its properties and validate our theoretical results. 

We invoke the matricized form \eqref{eq:sample_fonop} of the component-wise minimization problem for the sake of computational and storage efficiency. In particular, at the $k$-th block iterate, we compute the sample/batch tensor average $\tilde \X^k = \frac{1}{m_k}\sum_{l=1}^{m_k} \X^{k,l}$. Cycling through each component $i=1,\hdots,p$, we then compute the component-wise minimizer as the stationary matrix $\hat A_i^{k+1}$ satisfying the matrix equation
\begin{equation}\label{eq:batch_sample_fonop}
\tilde \X_{(i)}^k \Theta_i = \hat A_i^{k+1}(\Theta_i^T\Theta_i + \lambda I_r).
\end{equation}
Finally, we update the $i$-th component factor matrix $A_i^{k}$ using the relaxation step
\[
A_i^{k+1} = \alpha^{k,i} \hat A_i^{k+1} + (1-\alpha^{k,i}) A_i^{k}.
\]
In all our numerical experiments, we use step sizes of the form $\alpha^{k,i} = \frac{1}{k}$, i.e. $c^{k,i} = 1$ for $i=1,\hdots,p$ and  $k=1,2,\hdots$. 
\begin{example}[Convergence] 
In our first example, we apply the SALS algorithm to a 3-dimensional random tensor of size $(30,40,50)$ constructed from a deterministic tensor $\E[\X]$ with known decomposition. We generate random samples by perturbing each entry of $\E[\X]$ with a uniformly distributed noise $U\sim UNIF(-\delta,\delta)$. We measured the accuracy at each step via the residual $\sqrt{\E[\|\X-\db{\x}\|^2]}/\sqrt{\E[\|\X\|^2]}$ which can be computed exactly, since $E[\X]$ and $\E[\X^2]$ are known. \Cref{fig:example01} shows that when the noise parameter is zero and the stepsize rule is constant, we recover the quotient-linear performance of the regularized ALS method. In the presence of noise, however, the convergence rate decreases.

\begin{figure}[tbhp]
\centering
\subfloat[Semi-log plot of of error in decomposing $\E{[\X]}$ using ALS, noise parameter $\sigma=0$, $n_1=30, n_2=40, n_3=50$.]{\includegraphics[scale=0.65]{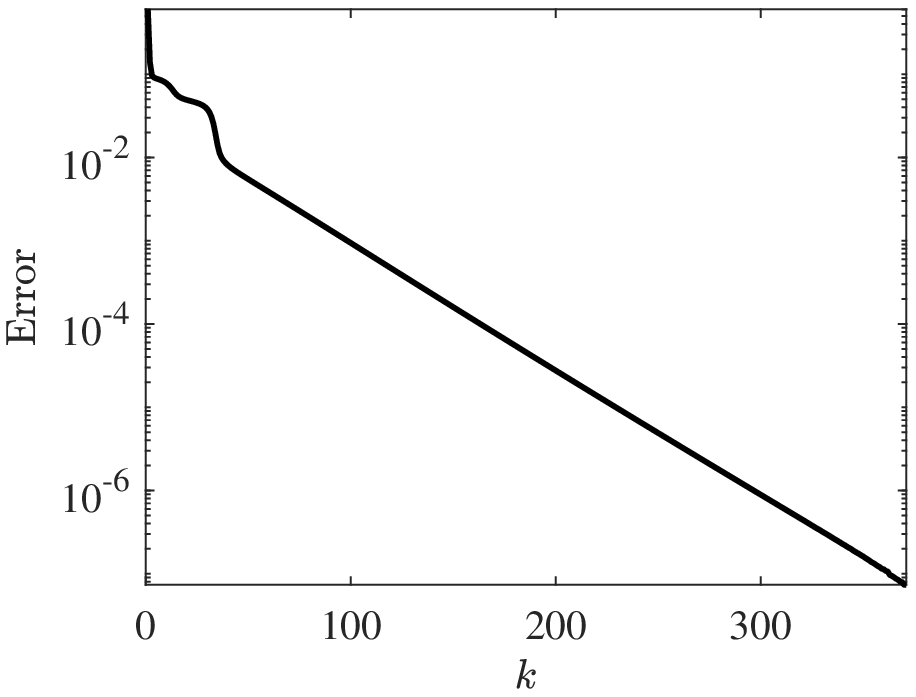}}\ \ \subfloat[Log-log plot of the error, using SALS with constant and variable stepsize with noise parameter $\delta=1$]{\includegraphics[scale=0.65]{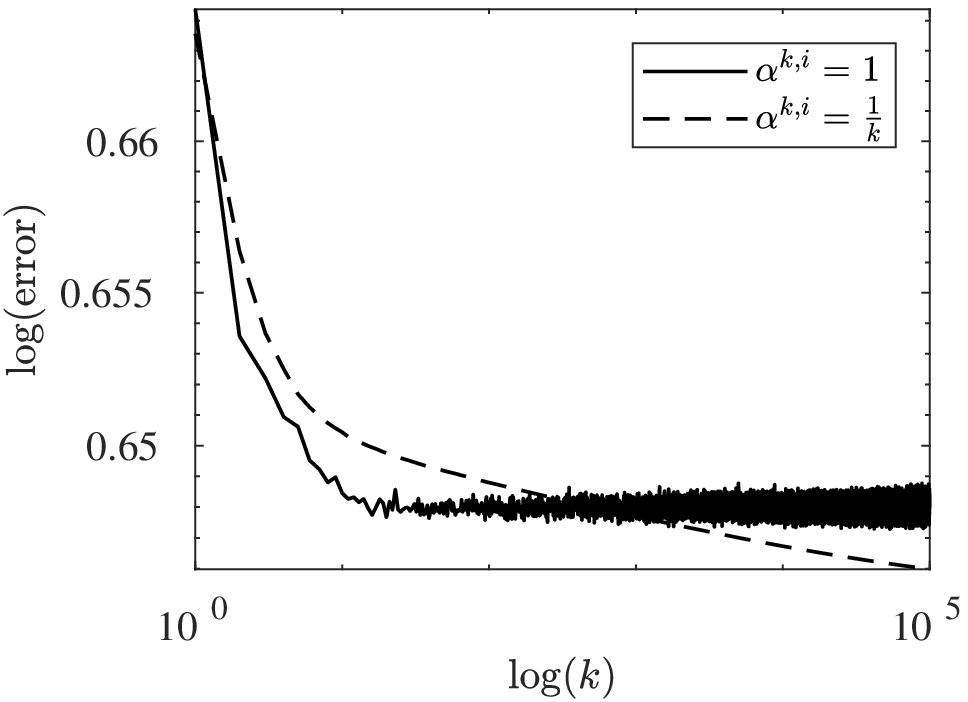}}
\caption{The influence of noise on the ALS iteration.}
\label{fig:example01}
\end{figure}
\end{example}

\subsubsection*{Computational Complexity} 

The computational effort of each block iteration of the SALS algorithm is determined by the cost of forming and solving Equation \eqref{eq:batch_sample_fonop} for each component. Recall that $r$ is the specified rank, $p$ the tensor's dimension, and $n=\sum_{i=1}^p n_i$ where $n_i$ is the size of the $i$-th vector component of each outer product in the expansion. Further, let $N=\prod_{i=1}^p n_i$ be the total number of entries in $\X$ and $\mathrm{nnz}(\X)$ be the number of non-zero entries in $\X$. It can then readily be seen that the cost of forming the coefficient matrix $\Theta_i^T\Theta_i + \lambda I_r$ is $O(pr^2 n)$, that of computing the matricized tensor times Khatri-Rao product (MTTKRP) $\tilde \X_{(i)}^k \Theta_i$ is $O(pr\cdot \mathrm{nnz}(\tilde \X_{(i)}^k))$, and that of solving the resulting dense symmetric system is $O(pr^3)$. If $r\ll N$, the computational cost of the MTTKRP dominates, especially as the density of $\tilde \X_{(i)}^k$ increases and hence $\mathrm{nnz}(\tilde \X_{(i)}^k)\rightarrow N$. 

It is often claimed that single-sample stochastic gradient methods, i.e. those using $m_k=1$ for $k=1,2,\hdots$, make more efficient use of samples than their batch counterparts, i.e. $m_k=m>1$ for $k=1,2,\hdots$. Below, we investigate this claim for the SALS method. For simplicity assume that each sample tensor $\X^l$ has a fixed fraction $\gamma \in [0,1]$ of zero entries, so that $\mathrm{nnz}(\X^l)=N(1-\gamma)$. Barring cancellation and assuming a uniform, independent distribution of nonzero entries, we estimate $\mathrm{nnz}(\tilde \X)=N(1-\gamma^m)$, where $m$ is the sample size. Adding the cost $O((m-1)N(1-\gamma)$ of forming the average $\tilde \X$ then gives the computational cost of each block iteration as
\begin{equation}\label{eq:per_iteration_cost}
C(m) = N(1-\gamma^m)+(m-1)N(1-\gamma).
\end{equation}
To compare the asymptotic efficiency of the single-sample SALS method with that of a batch method, we fix a computational budget $C_{\max}$ and investigate the estimated error achieved by each approach. Note that once $m$ is fixed, the number of iterations is determined by the budget as $k=C_{\max}/C(m)$. We first consider the case $m=1$ and let $\x_* = \argmin F(\x)$ be the overall minimizer. If we assume the convergence rate to be sublinear, i.e. the error satisfies
\[
\E\left[F(\x_k) - F(\x_*)\right] = O\left(\frac{1}{k^{\beta}}\right), \quad \text{ for some }\beta > 0,
\]
then the accuracy of the SALS method with a budget of $C_{\max}$ is on the order of
\begin{equation}\label{eq:efficiency_m_is_one}
E(1) = O\left(\left(\frac{C_{\max}}{C(1)}\right)^{-\beta}\right)=O\left(\left(\frac{C_{\max}}{\mathrm{nnz}(\tilde \X)}\right)^{-\beta}\right).
\end{equation}
Now suppose $m>1$. To obtain the best possible accuracy of the batch ALS methods (with budget $C_{\max}$) requires the optimal combination of batch size $m$ and number of iterations $k$. Let $\x_*^m = \argmin \tilde f_m(\x)$ be the batch minimizer, where $\tilde f_m(\x)=\frac{1}{m}\sum_{l=1}^m f(\x,\X^l)$, and let $\x_k$ be the terminal iterate. The overall error is given by, 
\begin{multline*}
\E\left[F(\x_k)-F(\x_*)\right] \leq \E\left[F(\x_k)-\tilde f_m(\x_k)\right] + \E\left[\tilde f_m(\x_k)-\tilde f_m(\x_*^m)\right] \\
+ \E\left[\tilde f_m(\x_*^m)-\tilde f_m(\x_*)\right] + \E\left[\tilde f_m(\x_*)-F(\x_*)\right],
\end{multline*}     
Since $\tilde f_m(\x_*^m)\leq \tilde f_m(\x^*)$, we can bound the error above by the sum of a sampling error $\E\left[F(\x_k)-\tilde f_m(\x_k)\right]+\E\left[\tilde f_m(\x_*)-F(\x_*)\right]$ and an optimization error $\E\left[\tilde f_m(\x_k)-\tilde f_m(\x_*^m)\right]$. Assuming the asymptotic behavior of the sampling error to be $O(1/\sqrt{m})$ and using the q-linear convergence rate $O(\rho^k)$ for some $\rho\in[0,1)$ of ALS (see \cite{Bezdek2003}) as an optimistic estimate, the total error takes the form $O(1/\sqrt{m}+\rho^k)$. Under the budget constraint, the error can thus be written as
\begin{equation}\label{eq:efficiency_m_more_than_one}
E(m) = O\left(\frac{1}{\sqrt{m}} + \rho^{\frac{C_{\max}}{C(m)}}\right).
\end{equation}
For large, high-dimensional tensors the number of iterations $C_{\max}/C(m)$ allowable under the budget will decrease rapidly as $m$ grows, so that the second term dominates the optimal error in Equation \eqref{eq:efficiency_m_more_than_one}, whereas smaller problems have an optimal error dominated by the Monte-Carlo sampling error $1/\sqrt{m}$. 

\begin{example}[Sparsity]
To compare the computational complexity of the single-sample SALS method, i.e. $m_k=1$, with that of the batch SALS method, i.e. $m_k=m$, we form a sequence of three-dimensional random sparse tensors with $n_1=30$, $n_2=40$, and $n_3=50$ and a fixed fraction $\gamma \in (0,1)$ of nonzero entries. For these tensors, the first and second moments are known and hence residuals can be computed accurately. In \cref{fig:example02a} we verify how an increase in the batch size $m$ deteriorates the sparsity of the tensor and lowers the sampling error at the predicted rate.

\begin{figure}[tbhp]
\centering
\subfloat[Deterioration of sparsity due to averaging.]{\includegraphics[scale=0.65]{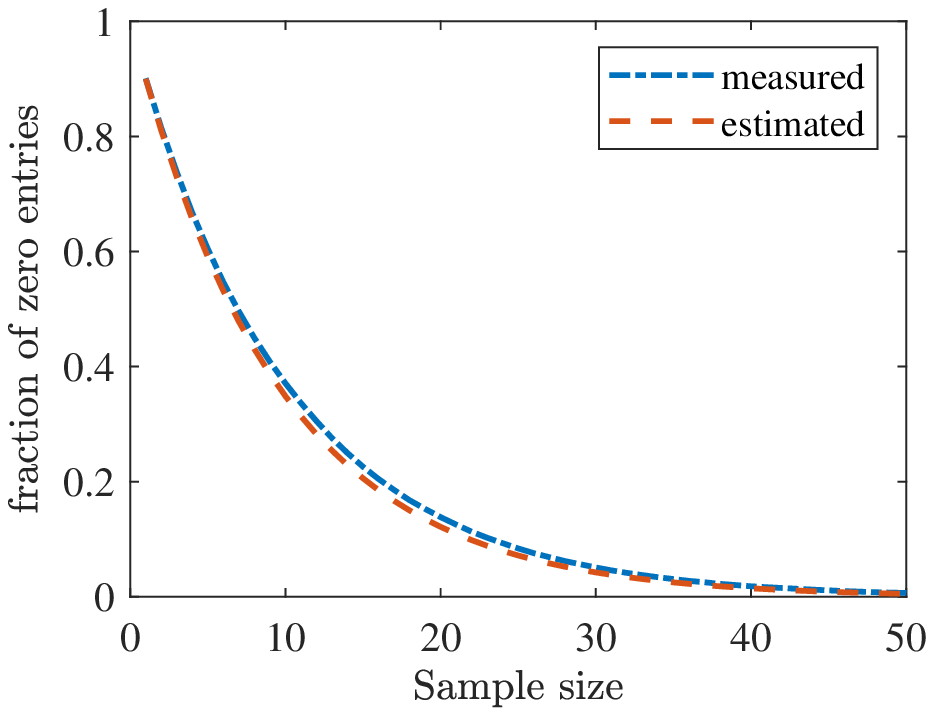}}\ \ \subfloat[Log-log plot of sampling error.]{\includegraphics[scale=0.65]{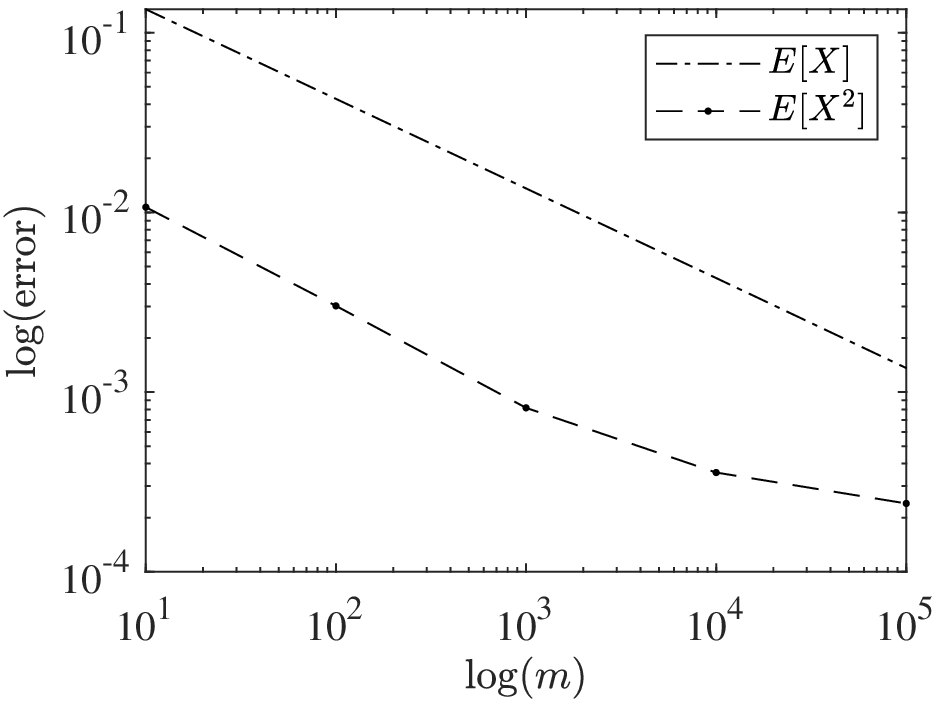}}
\caption{The effects of batch size $m$ on sparsity (left) and on the sampling error in the estimation of moments (right) for $(30,40,50)$ random tensor with sparsity parameter $\gamma=0.1$.}
\label{fig:example02a}
\end{figure}

\Cref{fig:example02b} shows the efficiency of the SALS algorithm with batch sizes $m=1,10,100$, and $1000$. To this end, we fixed a computational budget of 10k total samples and used this to determine the number of iteration steps based on the batch size $m$. Clearly the SALS method that uses a single sample for each block iteration makes much more efficient use of samples than variants with larger batch sizes.

\begin{figure}[tbhp]
\centering
\subfloat[Log-log plot of the residual error as a function of total number of samples.]{\includegraphics[scale=0.65]{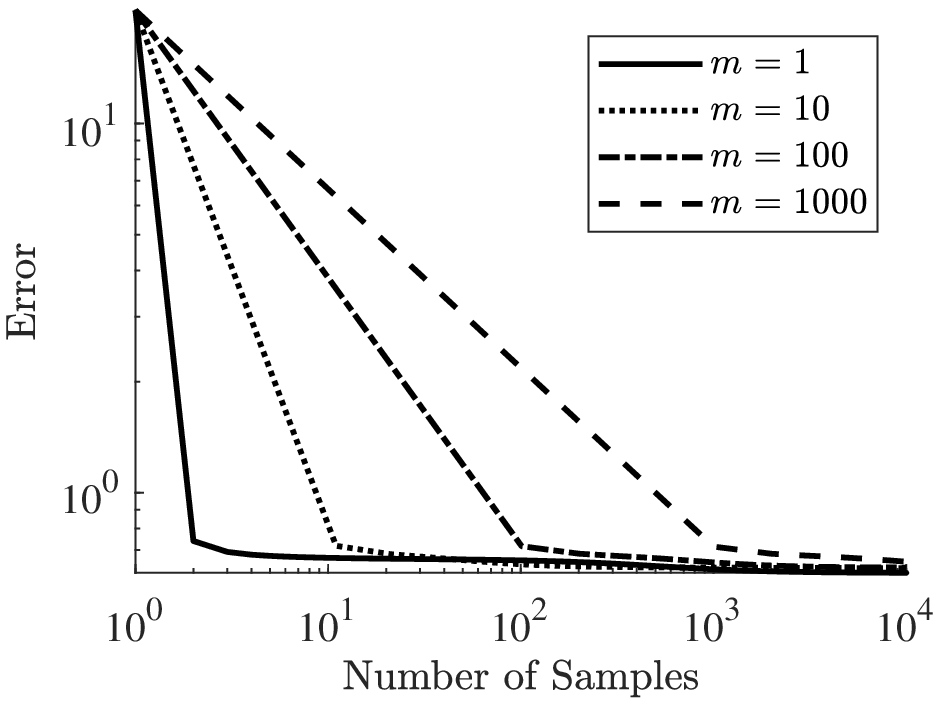}}\ \ \subfloat[Log-log plot of the residual error as a function of computational cost.]{\includegraphics[scale=0.65]{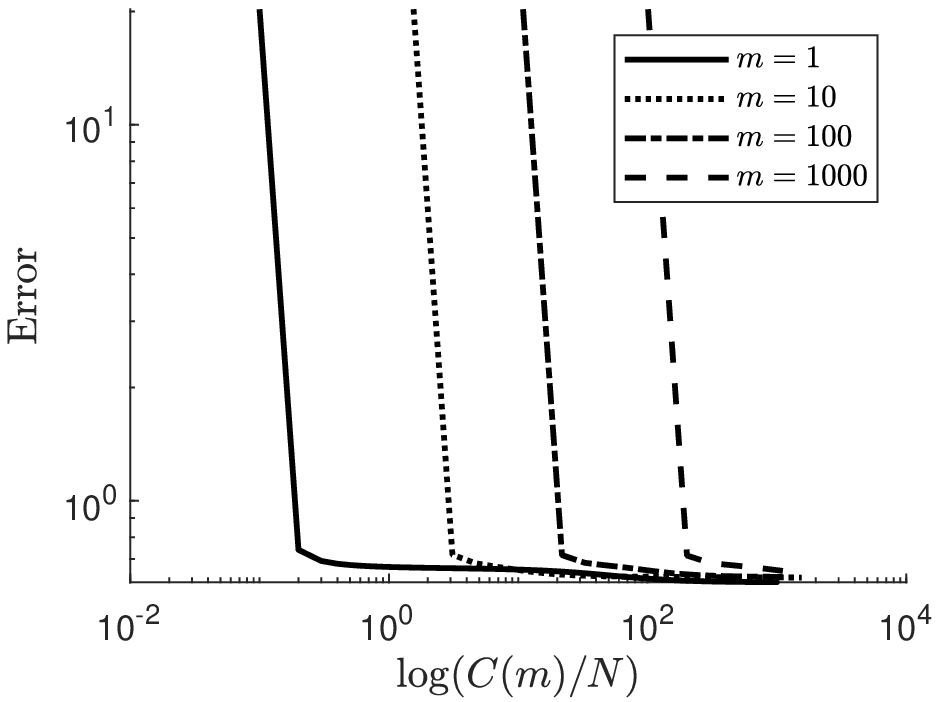}}
\caption{The efficiency of the SALS method for batch sizes $m=1,10,100,1000$.}
\label{fig:example02b}
\end{figure}

\end{example}

\section{Conclusion}\label{section:conclusion}

In this work, we showed that the SALS method is globally convergent and demonstrated the benefits of using stochastic-gradient-based approaches, especially in the decomposition of large, sparse tensors. In our analysis we focused on regularization in the Frobenius norm, and have not included a discussion on related proximal point algorithms or other regularization approaches (see e.g. \cite{Li2013b, Uschmajew2012, Xu2015}). Another interesting avenue of exploration relates to the choice of cost functional in tensor decomposition. Even though the quadratic structure afforded by the Frobenius norm is essential to the effectiveness of the ALS method, the standard statistical basis of comparison, namely expectation, could potentially be extended to other statistical metrics, such as those used in risk averse optimization methods. Finally, we foresee the application and extension of the SALS in the analysis and decomposition of real data sets.  

\newcommand{\arxiv}[1]{}

\bibliographystyle{siam}
\bibliography{sals_bib}

\end{document}